\documentclass[a4paper,twoside,11pt]{amsart}

\usepackage{amsfonts,amssymb,amscd,amsmath,enumerate}
\usepackage[greek, french, english]{babel}
\usepackage{graphicx}
\usepackage[all]{xy}
\usepackage{enumerate}
\usepackage{color}
\usepackage[T1]{fontenc}

\newtheorem{theorem}{Theorem}[section]
\newtheorem{proposition}[theorem]{Proposition}
\newtheorem{lemma}[theorem]{Lemma}
\newtheorem{corollary}[theorem]{Corollary}

\theoremstyle{definition}
\newtheorem{definition}[theorem]{Definition}

\newtheorem{example}[theorem]{Example}
\theoremstyle{remark}
\newtheorem{remark}[theorem]{Remark}

\newlength{\szer}

\def\A{{\mathbf A}}
\def\Aa{{\mathcal A}}

\def\N{{\mathbf N}}

\def\Pp{{\mathcal P}}

\def\Q{{\mathbf Q}}
\def\Qq{{\mathcal Q}}
\def\R{\mathbf R}

\def\Z{{\mathbf Z}}
\def\Z{\mathbf Z}

\def\int{\mathrm{int}}

\makeindex

\def\elem(#1,#2){  \{ \frac{#1}{\overline{\ #2\ }}\}  }

\begin{document}
\bibliographystyle{amsplain}
\setcounter{tocdepth}{2}
\title{EQUATIONS FOR FORMAL TORIC DEGENERATIONS}
\author{Bernard Teissier}
\address{Institut de Math\'ematiques de Jussieu - Paris Rive Gauche, UMR 7586 du CNRS, 
B\^atiment Sophie Germain, Case 7012,
75205 PARIS Cedex 13, France}

\email{bernard.teissier@imj-prg.fr}


\pagestyle{myheadings}
\markboth{\rm Bernard Teissier}{\rm Toric degeneration}
\renewcommand\rightmark{B. Teissier}
\renewcommand\leftmark{Equations for formal Toric degenerations}

\subjclass[2000]{14M25, 14E15, 14B05}

\dedicatory{ \textit {}}
\maketitle
\begin{abstract}
Let $R$ be a complete equicharacteristic noetherian local domain and $\nu$ a valuation of its field of fractions whose valuation ring dominates $R$ with trivial residue field extension. The semigroup of values of $\nu$ on $R\setminus \{0\}$ is not finitely generated in general. We produce equations in an appropriate generalized power series ring for the algebra encoding the degeneration of $R$ to the toric graded algebra  ${\rm gr}_\nu R$ associated to the filtration defined by $\nu$. We apply this to represent $\nu$ as the limit of a sequence of Abhyankar semivaluations (valuations on quotients) of $R$ with finitely generated semigroups. \end{abstract}
\maketitle
\tableofcontents

\begin{section}{Introduction}
For us a valuation is an inclusion of $R$ in a valuation ring $(R_\nu,m_\nu)$ of its field of fractions $K$. If the valuation is centered at $m$, which means that $m_\nu\cap R=m$, the dimension of the valuation is defined to be the transcendence degree of $R_\nu/m_\nu$ over $R/m$. A local uniformization of the valuation $\nu$ on $R$ is a regular local ring $R'$ which is essentially of finite type over $R$ and dominated by $R_\nu$.\par
We continue here the program, presented in \cite{Te1}, \cite{Te3} and \cite{Te5}, to prove local uniformization for valuations on equicharacteristic excellent local domains $(R,m)$ with an algebraically closed residue field. Here is its structure:
\begin{enumerate}
\item Reduce the general case to the case of \textit{rational} valuations, those with trivial residue field extension: $R_\nu/m_\nu=R/m$;
\item Reduce the case of rational valuations to the case where $R$ is complete:
\item Assuming that $R$ is complete, reduce the case of a rational valuation to that of Abhyankar valuations (rational rank=${\rm dim}R$);  
\item Assuming that $R$ is complete, prove the uniformization theorem for Abhyankar valuations using embedded resolution of affine toric varieties, which is blind to the characteristic, applied to the toric graded ring ${\rm gr}_\nu R$ associated to the filtration defined by the valuation.
\end{enumerate}
Step 1) was dealt with in \cite[Section 3.6]{Te1}. Step 2), except in the cases of rank one or Abhyankar valuations, is waiting for the proof of a conjecture about the extension of valuations of $R$ to its completion  (see \cite[*Proposition 5.19*]{Te1}, \cite[Conjecture 1.2]{HOST2} and \cite[Problem B]{Te5}), which is in progress. Step 4) was made in \cite{Te3}. The essential remaining hurdle is step 3).\par The main result of this paper is the production of a sequence of Abhyankar semivaluations $\nu_B$ of $R$ approximating more and more closely our rational valuation $\nu$ in an embedded manner. We believe that they have the property that appropriate toric embedded uniformizations as in \cite{Te3} of the $\nu_B$ close enough to $\nu$ will also uniformize $\nu$. Since we think that the approximation result is of independent interest we make this paper available without waiting for the final proof.
\par\medskip
A valuation $\nu$ on a subring $R$ of $K$ determines a filtration of $R$ by the ideals $$\Pp_\phi(R)=\{x\in R/\nu(x)\geq \phi\}\  {\rm and}\ \  \Pp^+_\phi (R)=\{x\in R/\nu(x)> \phi\}.$$
We note that this is not in general a filtration in the usual sense because the totally ordered value group $\Phi$ of $\nu$ may not be well ordered so that an element of $\Phi$ may not have a successor. However, the successor ideal  $ \Pp^+_\phi (R)$ is well defined and if $\phi$ has a successor $\phi^+$ we have $\Pp^+_\phi(R)=\Pp_{\phi^+}(R)$.\par\noindent The graded ring ${\rm gr}_\nu R=\bigoplus_{\phi\in \Phi}\Pp_\phi(R)/\Pp^+_\phi(R)$ associated to the $\nu$-filtration on $R$ is the graded $k$-subalgebra of ${\rm gr}_\nu K$ whose non zero homogeneous elements have degree in the sub-semigroup\footnote{I follow a tradition of calling semigroup of values of a valuation what is actually a monoid.} $\Gamma=\nu(R\setminus\{0\})\cup\{0\}$ of $\Phi$. In this text we shall deal with subrings $R$ of the valuation ring $R_\nu$, so that ${\rm gr}_\nu R$ is a graded subalgebra of ${\rm gr}_\nu R_\nu$ since by definition $\Pp_\phi(R)=\Pp_\phi(R_\nu)\cap R$, and $\Gamma$ is contained in the non-negative part $\Phi_{\geq 0}$ of $\Phi$.\par In view of the defining properties of valuation rings, for all $\phi\in\Phi$ any two elements of $\Pp_\phi(R_\nu)\setminus \Pp^+_\phi(R_\nu)$ differ by multiplication by a unit of $R_\nu$ so that their images in $\Pp_\phi(R)/\Pp^+_\phi(R)$ differ by multiplication by a non zero element of $k_\nu=R_\nu/m_\nu$. Thus, the homogeneous components of ${\rm gr}_\nu R_\nu$ are all one dimensional vector spaces over $k_\nu$.\par \textit{If the valuation $\nu$ is rational}, the same is true of the non zero homogeneous components of ${\rm gr}_\nu R$ as $k=R/m$ vector spaces. Then one has a toric description of ${\rm gr}_\nu R$ as follows:\par\noindent
Since $R$ is noetherian the semigroup $\Gamma$ is well ordered and so has a minimal system of generators $\Gamma=\langle \gamma_1,\ldots ,\gamma_i,\ldots\rangle$, where $\gamma_{i+1}$ is the smallest non zero element of $\Gamma$ which is not in the semigroup $\langle\gamma_1,\ldots, \gamma_i\rangle$. generated by its predecessors. We emphasize here that the $\gamma_i$ are indexed by an ordinal $I\leq \omega^h$, where $h$ is the rank (or height) of the valuation. The graded $k$-algebra ${\rm gr}_\nu R$ is then generated by homogeneous elements $(\overline\xi_i)_{i\in I}$ with ${\rm deg}\overline\xi_i=\gamma_i$ and we have a surjective map of graded $k$-algebras
$$k[(U_i)_{i\in I}]\longrightarrow {\rm gr}_\nu R,\ \ \ U_i\mapsto \overline\xi_i,$$ 
where $k[(U_i)_{i\in I}]$ is the polynomial ring in variables indexed by $I$, graded by giving $U_i$ the degree $\gamma_i$.\par\noindent Its kernel  is generated by binomials $(U^{m^\ell}-\lambda_\ell U^{n^\ell})_{\ell\in L},\ \lambda_\ell\in k^*$, where $U^m$ represents a monomial in the $U_i$'s. These binomials correspond to a generating system of relations between the generators $\gamma_i$ of the semigroup, and for all practical purposes we can think of ${\rm gr}_\nu R$ as the semigroup algebra over $k$ of the semigroup $\Gamma$. This is asserted in \cite[Proposition 4.7]{Te1} and made more precise in subsection \ref{gral} below .\textit{We assume that the set $(U^{m^\ell}-\lambda_\ell U^{n^\ell})_{\ell\in L}$ is such that none of the binomials is in the ideal generated by the others.} Since the semigroup of weights is well ordered and there are only finitely many binomials of a given weight, this can be achieved by a transfinite cleaning process eliminating at each step the binomials of least weight belonging to the ideal generated by binomials $U^{m^\ell}-\lambda_\ell U^{n^\ell}$, necessarily of smaller weight.\par
An important ingredient in the approach presented in \cite{Te1} and \cite{Te3} and continued here is the fact that by a result of Pedro Gonz\'alez P\'erez and the author (see \cite{Te1}, 6.2 and \cite{GP-T1}), we have embedded resolutions by birational toric maps of affine toric varieties (of finite embedding dimension) such as ${\rm Spec}k[t^\Gamma]$ when the semigroup $\Gamma$ is finitely generated.\par\noindent By general results in \cite[\S 2.3]{Te1} the ring $R$ is a deformation of ${\rm gr}_\nu R$ and the basic observation is that for rational valuations, when ${\rm gr}_\nu R$ is finitely generated, not only is the faithfully flat specialization of $R$ to ${\rm gr}_\nu R$ equidimensional, so that we have equality in Abhyankar's inequality, but the total space of this toric degeneration is "equiresolvable at $\nu$" in the sense that there exist generators $(\xi_i)_{i\in I}$ of the maximal ideal of $R$ whose initial forms $(\overline\xi_i)_{i\in I}$ are generators of  the $k$-algebra ${\rm gr}_\nu R$ and some birational toric maps \textit{in the coordinates $(\xi_i)_{i\in I}$} which provide, in the coordinates $(\overline\xi_i)_{i\in I}$, an embedded resolution of singularities of the affine toric variety ${\rm Spec}{\rm gr}_\nu R$, and when applied to the generators $\xi_i$ of the maximal ideal of $R$ and localized at the point (or center) picked by $\nu$ provide an embedded local uniformization of the valuation $\nu$.\par Of course, in the general case where ${\rm gr}_\nu R$ is not finitely generated even after replacing $R$ by some birational modification along $\nu$, we must adapt this notion of "$\nu$-equiresolvability" since ${\rm Spec}{\rm gr}_\nu R$ does not even have a resolution of singularities, embedded or not.\par\noindent By a theorem of Piltant (see \cite{Te1}, proposition 3.1) we know however that its Krull dimension ${\rm dim}{\rm gr}_\nu R$ is the rational rank $r(\nu)$ of the valuation.\par\noindent For rational valuations of a noetherian local domain Abhyankar's inequality then reads as the following inequality of Krull dimensions:
$${\rm dim}{\rm gr}_\nu R\leq {\rm dim}R.$$ Those for which equality holds, called Abhyankar valuations, were the subject of \cite{Te3}, where it was shown that their graded algebra was quasi-finitely generated\footnote{This means that there is a local $R$-algebra $R'$ essentially of finite type dominated by the valuation ring $R_\nu$ (a birational $\nu$-modification) such that ${\rm gr}_\nu R'$ is finitely generated. For more on finite and non finite generatedness of ${\rm gr}_\nu R$, see \cite{C}.}, and that they could be uniformized in an embedded manner\footnote{A proof of local uniformization for Abhyankar valuations of algebraic function fields, with assumptions on the base field, had appeared in \cite{K-K} and a more general one in \cite{C2}. See also \cite{T}.}.\par The proof reduced the excellent equicharacteristic case to the complete equicharacteristic case using the good behavior of Abhyankar valuations with respect to completion. The case of complete equicharacteristic local domains was then reduced to a combinatorial problem using a valuative version of Cohen's structure theorem to obtain equations describing the specialization of $R$ to ${\rm gr}_\nu R$.\par\noindent
The geometric interpretation of the first part of the valuative Cohen theorem is that after embedding ${\rm Spec}R$ in an affine space $\A^{\vert I\vert}(k)$ by the generators $(\xi_i)_{i\in I}$, the valuation $\nu$ becomes the trace on ${\rm Spec}R$ of a \textit{monomial} valuation (a.k.a. weight in the sense of \cite[Definition 2.1]{Te3}) on the the ambient possibly infinite dimensional affine space.\par\noindent 
The interpretation of the second part is that the image of the embedding in question is defined by a (possibly infinite) set of equations which constitutes an \textit{overweight deformation} (see definition \ref{overweight} below) of a set of binomial equations defining the generalized affine toric variety ${\rm Spec}{\rm gr}_\nu R$. The classical Cohen structure theorem presents a complete equicharacteristic noetherian local ring $R$ as a quotient of a power series ring in finitely many variables by an ideal generated by overweight deformations of the homogeneous polynomiala defining the associated graded ring ${\rm gr}_mR$ as a quotient of a polynomial ring.

\begin{example}\label{ex1} Let $R$ be the algebra of an algebroid plane branch over an algebraically closed field $k$. That is, a complete equicharacteristic one dimensional local domain with residue field $k$. The integral closure is $R\subset k[[t]]$ and $R$ has only one valuation $\nu$: the $t$-adic valuation. The set $I=(\gamma_0,\gamma_1,\ldots ,\gamma_g)$ of degrees of a minimal set of generators $(\overline\xi_i)_{i\in I}$ of the $k$-algebra ${\rm gr}_\nu R\subset k[t]$, ordered by weight, is in characteristic zero in order preserving bijection with the set $(\beta_0, \beta_1, \ldots ,\beta_g)$ of characteristic Puiseux exponents. For positive characteristc see \cite[Example 6.2]{Te1} or \cite{Te4}. Representatives $\xi_i\in R$ of the $(\overline\xi_i)_{i\in I}$ embed the curve in $\A^{g+1}(k)$ as a complete intersection defined by overweight deformations of the binomials defining the monomial curve ${\rm Spec}k[t^\Gamma]$. (read \cite[Example 5.69]{Te1} in the formal power series context). \end{example}

When strict inequality holds, which implies that the $k$-algebra ${\rm gr}_\nu R$ is not finitely generated,  $\nu$-equiresolvability should mean that some \textit{partial} toric resolutions of the singularities of the toroidal scheme ${\rm Specgr}_\nu R$, affecting only finitely many of the $(\overline\xi_i)_{i\in I}$, when applied to the elements $\xi_i\in R$ instead of the $\overline\xi_i\in {\rm gr}_\nu R$, should uniformize the valuation $\nu$ on ${\rm Spec}R$.\par The problem is to find the appropriate finite subset of the set $L$ of binomial equations written above.  It should generate a prime ideal in a finite set of variables and have the property that some of the embedded toric resolutions of this prime binomial ideal uniformize the valuation. Theorem \ref{AbhApprox} at the end of this paper at least provides candidates. \par
\end{section}
\begin{section}{Valued complete equicharacteristic noetherian local rings and the valuative Cohen theorem}
\begin{subsection}{The graded algebra of a rational valuation is a twisted semigroup algebra}\label{gral}
Let $\nu$ be a rational valuation centered in a noetherian local domain $R$ with residue field $k$. denoting by $\overline k$ the algebraic closure of $k$, given a minimal system of generators $(\overline\xi_i)_{i\in I}$ of the $k$-algebra ${\rm gr}_\nu R$, a simple algebraic argument based on the fact that the group $\Phi$ has a finite rational rank shows that there exist elements $\rho_i\in (\overline k)^*$ such that the map of $\overline k$-algebras $\overline k\otimes_k{\rm gr}_\nu R\to \overline k[t^\Gamma]$ determined by $1\otimes_k\overline\xi_i \mapsto \rho_i t^{\gamma_i}$ is an isomorphism of graded algebras. By \cite[ Theorems 17.1 and 21.4]{G} we know that the Krull dimension of $k[t^\Gamma]$ is the rational rank $r$ of $\nu$ and by Piltant's Theorem that ${\rm gr}_\nu R$ is also an $r$-dimensional domain. Considering the surjective composed map
$$\overline k[(U_i)_{i\in I}]\longrightarrow \overline k[t^\Gamma],\ \ U_i\mapsto \rho_it^{\gamma_i},$$all that is needed is that the $U_i=\rho_i$ are non zero solutions of the system of binomial equations $(U^{m^\ell}-\lambda_\ell U^{n^\ell}=0)_{\ell\in L}$ which generate the kernel. There exists a finite set $F\subset I$ such that the $(\gamma_i)_{i\in F}$ rationally generate the value group $\Phi$ of our valuation. The binomial equations involving only the variables $(U_i)_{i\in F}$ are finite in number and have non zero solutions $\rho_i\in (\overline k)^*$ by the nullstellensatz and quasi-homogeneity. All the other $\gamma_i$ are rationally dependent on those. The result then follows by a transfinite induction in the well ordered set $I$: assuming the algebraicity to be true for binomial equations corresponding to the relations between the $(\gamma_j)_{j\in F\cup\{j\in I\vert j<i\}}$ and choosing non zero algebraic solutions $\rho_j$, we substitute $\rho_j$ for $U_j$ in the binomials involving the variables $(U_j)_{j\in F\cup\{j\in I\vert j<i\}}$ and $U_i$ to obtain binomial equations in $U_i$ with coefficients in $k((\rho_j)_{j\in F\cup\{j\in I\vert j<i\}})$. By rational dependence we know there is at least one such equation, of the form $U_i^{n_i}-m_i(\rho)=0$ where $m_i(\rho)$ is a non zero Laurent term in the $(\rho_j)_{j\in F}$, and in $k((\rho_j)_{j\in F\cup\{j\in I\vert j<i\}})[U_i,U_i^{-1}]$ the ideal generated by the equations coming from all binomials involving $U_i$ is principal so that up to multiplication by a root of unity we have a well defined solution $\rho_i\in k^*$ and this gives the result. So we have proved the following slight generalization of known results (compare with \cite[Proposition 4.7]{Te1} and, for the case of submonoids of $\Z^r$, \cite[Exercise 4.3]{B-G}):
\begin{proposition}\label{smgrp}Let $\Gamma$ be the semigroup of values of a rational valuation $\nu$ on a noetherian local domain $R$ with residue field $k$. Given a minimal set $(\overline\xi_i)_{i\in I}$ of homogeneous generators of the $k$-algebra ${\rm gr}_\nu R $, with ${\rm deg}\xi_i=\gamma_i$, the map which to an isomorphism $\rho$ of graded $\overline k$-algebras $$\rho\colon \overline k\otimes_k{\rm gr}_\nu R \simeq \overline k[t^\Gamma]$$
associates the family of elements $(\rho (1\otimes_k\overline\xi_i)t^{-\gamma_i})_{i\in I}$ of $(\overline k)^*$ defines a bijection from the set of such isomorphisms to the set of $\overline k$-rational points of  ${\rm Specgr}_\nu R$ all of whose coordinates in ${\rm Spec}\overline k[(U_i)_{i\in I}]$ are non zero; a generalized torus orbit. This set is not empty.
\end{proposition}
The graded algebra ${\rm gr}_\nu R$ is a \textit{twisted semigroup algebra} in the sense of \cite[Definition 10.1]{K-M}. In particular, if $k$ is algebraically closed, there are many isomorphisms of graded $k$-algebras from ${\rm gr}_\nu R$ to $k[t^\Gamma]$.\par\noindent
 For a different viewpoint on the same result, and some interesting special cases, see \cite{B-N-S}. 
\subsection{Overweight deformations of prime binomial ideals (after \cite[Section 3]{Te3})}\label{ow}
Let $w\colon \N^I\rightarrow \Phi_{\geq 0}$ be a semigroup map defining a weight on the variables $(u_i)_{i\in I}$ of a polynomial or power series ring over a field $k$, with values in the positive part $\Phi_{\geq 0}$ of a totally ordered group $\Phi$ of finite rational rank. If we are dealing in finitely many variables, the semigroup of values taken by the weight is well-ordered by a theorem of B. Neumann in \cite[Theorem 3.4]{N}. \par Let us consider the power series case, in finitely many variables, and the ring $S=k[[u_1,\ldots ,u_N]]$. Consider the filtration of $S$ indexed by $\Phi_{\geq 0}$ and determined by the ideals $\Qq_\phi$ of elements of weight $\geq \phi$, where the weight of a series is the minimum weight of a monomial appearing in it. Defining similarly $\Qq^+_\phi$ as the ideal of elements of weight $>\phi$, the graded ring associated to this filtration is the polynomial ring $$
\bigoplus_{\phi\in \Phi_{\geq 0}}\Qq_\phi/\Qq^+_\phi=k[U_1,\ldots ,U_N],$$
with $U_i={\rm in}_wu_i$, graded by ${\rm deg}U_i=w(u_i)$.

\begin{definition}\label{overweight}
Given a weight $w$ as above, a (finite dimensional) overweight deformation is the datum of a prime binomial ideal $(u^{m^\ell}-\lambda_\ell u^{n^\ell})_{1\leq \ell \leq s},\ \lambda_\ell\in k^*$, of $S=k[[u_1,\ldots ,u_N]]$ such that the vectors $m^\ell-n^\ell\in \Z^N$ generate the lattice of relations between the $\gamma_i=w(u_i)$, and of series 
$$\begin{array}{lr}
F_1=u^{m^1}-\lambda_1 u^{n^1}+\Sigma_{w(p)>w(m^1)} c^{(1)}_pu^p\\
F_2=u^{m^2}-\lambda_2 u^{n^2}+\Sigma_{w(p)>w(m^2)} c^{(2)}_pu^p\\
.....\\
F_\ell=u^{m^\ell}-\lambda_\ell u^{n^\ell}+\Sigma_{w(p)>w(m^\ell)} c^{(\ell)}_pu^p\\
.....\\
F_s=u^{m^s}-\lambda_s u^{n^s}+\Sigma_{w(p)>w(m^s)} c^{(s)}_pu^p\\
\end{array}\leqno{(OD)}
$$ in $k[[u_1,\ldots ,u_N]]$ such that, with respect to the monomial order determined by $w$, they form a standard basis for the ideal which they generate: their initial forms generate the ideal of initial forms of elements of that ideal.\par\noindent 
Here we have written $w(p)$ for $w(u^p)$ and the coefficients $c^{(\ell)}_p$ are in $k$.\par
The dimension of the ring $R=k[[u_1,\ldots ,u_N]]/(F_1,\ldots ,F_s)$ is then equal to the dimension of $k[[u_1,\ldots ,u_N]]/(u^{m^1}-\lambda_1 u^{n^1},\ldots ,u^{m^s}-\lambda_s u^{n^s})$, which is the rational rank of the group generated by the weights of the $u_i$.\par\noindent
The structure of overweight deformation endows the ring $R$ with a rational valuation. The valuation of a nonzero element of $R$ is the maximum weight of its representatives in $k[[u_1,\ldots ,u_N]]$. It is a rational Abhyankar valuation.
\end{definition}
\end{subsection}

\begin{subsection}{A reminder about the valuative Cohen theorem}\label{reminder} When one wishes to compare an equicharacteristic noetherian local domain $R$ endowed with a rational valuation $\nu$ with its associated graded ring ${\rm gr}_\nu R$, the simplest case is when $R$ contains a field of representatives $k$ and the $k$-algebra ${\rm gr}_\nu R$ is finitely generated. After re-embedding ${\rm Spec} R$ in an affine space where ${\rm Spec}{\rm gr}_\nu R$ can be embedded one can again, at least if $R$ is complete, write down equations for the faithfully flat degeneration of ${\rm Spec}R$ to ${\rm Spec}{\rm gr}_\nu R$ within that new space (see \cite[\S 2.3]{Te1}) and use their special form as overweight deformations to prove local uniformization for the valuation $\nu$ as explained in \cite[\S 3]{Te3}. However, without some strong finiteness assumptions on the ring $R$ and the valuation $\nu$, one cannot hope for such a pleasant situation. The valuative Cohen theorem shows that when $R$ is complete one can obtain a similar state of affairs without any other finiteness assumption than the noetherianity of $R$.\par
Let $\nu$ be a rational valuation on a complete equicharacteristic noetherian local domain $R$. Since $R$ is noetherian, the semigroup of values $\Gamma=\nu(R\setminus\{0\})$ is countable and well ordered, and has a minimal set of generators $(\gamma_i)_{i\in I}$ which is indexed by an ordinal $I\leq \omega^{h(\nu)}$ according to \cite[Appendix 3, Proposition 2]{Z-S}. Consider variables $(u_i)_{i\in I}$ in bijection with the $\gamma_i$ and the $k$-vector space of all sums $\Sigma_{e\in E}d_eu^e$ where $E$ is any set of monomials in the $u_i$ and $d_e\in k$.  Since the values semigroup $\Gamma$ is combinatorially finite this vector space is in fact, with the usual multiplication rule, a $k$-algebra $\widehat{k[(u_i)_{i\in I}]}$, which we endow with a weight by giving $u_i$ the weight $\gamma_i$. Combinatorial finiteness means that there are only finitely many monomials with a given weight, and we can enumerate them according to the lexicographic order of exponents. Thus, we can embed the set of monomials $u^m$ in the well ordered lexicographic product $\Gamma\times \N$, where the finite set of monomials with a given weight is lexicographically ordered (see \cite[\S 4]{Te1}). Combinatorial finiteness also implies that the initial form of every series with respect to the weight filtration is a polynomial so that:\par\noindent \textit{ The graded algebra of $\widehat{k[(u_i)_{i\in I}]}$ with respect to the weight filtration is the polynomial algebra $k[(U_i)_{i\in I}]$ with $U_i={\rm in}_wu_i$, graded by giving $U_i$ the weight $\gamma_i$.}\par
The $k$-algebra $\widehat{k[(u_i)_{i\in I}]}$ is endowed with a monomial valuation given by the weight $w(\Sigma_{e\in E}d_eu^e)={\rm min}_{d_e\neq 0}w(u^e) $. This valuation is rational since all the $\gamma_i$ are $>0$. Note that $0$ is the only element with value $\infty$ because here $\infty$ is an element larger than any element of $\Gamma$. With respect to the "ultrametric" given by $u(x,y)=w(y-x)$, the algebra $\widehat{k[(u_i)_{i\in I}]}$ is spherically complete\footnote{A \emph{pseudo-convergent, {\rm or} pseudo-Cauchy} sequence of elements of $\widehat{k[(u_i)_{i\in I}]}$ is a sequence $(y_\tau)_{\tau\in T}$ indexed by a well ordered set $T$ without last element, which satisfies the condition that whenever $\tau<\tau' <\tau"$ we have $w(y_{\tau'}-y_\tau)<w(y_{\tau"}-y_{\tau'})$ and an element $y$ is said to be a \emph{pseudo-limit} of this pseudo-convergent sequence if $w(y_{\tau'}-y_\tau)\leq w(y-y_\tau)$ for $\tau ,\tau'\in T,\ \tau<\tau'$. Spherically complete means that every pseudo-convergent sequence has a pseudo-limit.} by \cite[Theorem 4.2]{Te3}.
\begin{lemma}\label{appear} Let $A,B$ be two sets of (not necessarily distinct) monomials. If a monomial $u^c$ appears infinitely many times among the monomials $u^{a+b}$ with $u^a\in A,\ u^b\in B$, then at least one decomposition $c=a+b$ must appear infinitely many times. Thus, some $u^a$ or $u^b$ with $c=a+b$ must appear infinitely many times in $A$ or $B$ respectively.
\end{lemma}
\begin{proof} The combinatorial finiteness of the semigroup $\Gamma$ implies that the number of posssible distinct decompositions $c=a+b$ is finite. Thus, at least one of them must occur infinitely many times. \end{proof}
\begin{definition}\label{closure} Let $J$ be an ideal in $\widehat{k[(u_i)_{i\in I}]}$. The closure of $J$ with respect to the ultrametric $w$ is the set of elements of $\widehat{k[(u_i)_{i\in I}]}$ which are a transfinite sum of elements of $J$. That is, the set of elements which can be written $y=\sum_{\tau\in T}(y_{\tau+1}-y_\tau)$ where $T$ is a well ordered set, $(y_\tau)_{\tau\in T}$ is a pseudo-convergent sequence and for each $\tau\in T$ we have $y_{\tau+1}-y_\tau\in J$.
\end{definition}
\begin{proposition}\label{closideal}The closure of an ideal of $\widehat{k[(u_i)_{i\in I}]}$ is an ideal.
\end{proposition}
\begin{proof} This follows from the fact that a transfinite series exists in $\widehat{k[(u_i)_{i\in I}]}$ if any given monomial appears only finitely many times. Since that must be the case for the sum $y$ of the definition, it is also the case for the sum or difference of two such elements, and for the product with an element of the ring. All the terms of such sums are in $J$.

\end{proof}
\begin{proposition}\label{h(u)} For any series $h(u)\in\widehat{k[(u_i)_{i\in I}]}$ which is without constant term, the series $\sum_{a=0}^\infty h(u)^a$ converges in $\widehat{k[(u_i)_{i\in I}]}$.
\end{proposition}
\begin{proof}Since there is no constant term, we have to prove that any non constant monomial can appear only finitely many times in the series. If the value group is of rank one, the result is clear by the archimedian property. We proceed by induction on the rank. If $\Psi_1$ is the largest non trivial convex subgroup of the value group $\Phi$, we can write $h(u)=h_1(u)+h_2(u)$ where $h_1(u)$ contains all those terms in $h(u)$ which involve only variables with weight in $\Psi_1$. Then $h(u)^a=\sum_{i=0}^a{{a}\choose{i}}h_1(u)^ih_2(u)^{a-i}$. If a monomial appears infinitely many times in the $(h(u)^a)_{a\in \N}$ then applying lemma \ref{appear} to the sets of monomials appearing in the $(h_1(u)^i)_{i\in \N}$ and $(h_2(u)^j)_{j\in \N}$ respectively shows that some monomial must appear infinitely many times in the collection of series $(h_1(u)^i)_{i\in \N}$, which we exclude by the induction hypothesis applied to $\Psi_1$, or some monomial must appear infinitely many times in the collection of series  $(h_2(u)^j)_{j\in \N}$, which is impossible because the semigroup $\Phi_+\setminus \Psi_1$ is archimedian since $\Phi/\Psi_1$ is.
\end{proof}
\begin{corollary}\label{local} The ring $\widehat{k[(u_i)_{i\in I}]}$ is local, with maximal ideal $\hat m$ consisting of series without constant term and residue field $k$. This ideal is the closure of the ideal generated by the $(u_i)_{i\in I}$.
\end{corollary}
Indeed the proposition shows that the non invertible elements of the ring are exactly those without constant term, which form an ideal. The second assertion follows from the fact that any series without constant term is a transfinite sum of $w$-isobaric polynomials in the $u_i$.
\begin{proposition}\label{hensel}The local ring $\widehat{k[(u_i)_{i\in I}]}$ is henselian.
\end{proposition}
\begin{proof} This follows from the fact that it is spherically complete, that the value group of $w$ is of finite rational rank, and the characterization of henselian rings by the convergence of pseudo-convergent sequences of \'etale type found in \cite[Proposition 3.7]{dF-T}.
\end{proof}
\begin{corollary}\label{algebraic} The henselization of the localization $k[(u_i)_{i\in I}]_{\hat m\cap k[(u_i)_{i\in I}]}$ of the polynomial ring $k[(u_i)_{i\in I}]$ at the maximal ideal which is the ideal of polynomials without constant term, is contained in $\widehat{k[(u_i)_{i\in I}]}$. 
\end{corollary}
\begin{proof} By proposition \ref{h(u)}, any element of $k[(u_i)_{i\in I}]$ which is not in the maximal ideal is invertible in $\widehat{k[(u_i)_{i\in I}]}$, which is henselian.
\end{proof}
\begin{corollary}\label{substitute} If the series $h(u)$ has no constant term, given any other series $g(u)\in \widehat{k[(u_i)_{i\in I}]}$ and a monomial $u^m$, writing $g(u)=\sum_{a\in \N}g_a(u)(u^m)^a$, where no term of a $g_a(u)$ is divisible by $u^m$, the substitution $\sum_{a\in \N}g_a(u)h(u)^a$ gives an element of $\widehat{k[(u_i)_{i\in I}]}$: in this ring, we can substitute a series without constant term for any monomial, and in particular for any variable.\end{corollary}
\begin{proof} Since any monomial appears finitely many times in $g(u)$, by combinatorial finiteness, any monomial can appear only finitely many times among the series $(g_a(u))_{a\in\N}$.\par\noindent Should a monomial appear infinitely many times among the terms of the sum $\sum_{a\in \N}g_a(u)h(u)^a$, by lemma \ref{appear} applied to the set of monomials appearing in the $(g_a(u))_{a\in\N}$ and the set of monomials appearing in the $(h(u)^a)_{a\in\N\setminus\{0\}}$, a monomial would have to appear infinitely many times among the $h(u)^a$ and this would contradict proposition \ref{h(u)}.
\end{proof}
\textit{In conclusion, $\widehat{k[(u_i)_{i\in I}]}$ is regular in any reasonable sense and has almost all of the properties of the usual power series rings except of course noetherianity if $\Gamma$ is not finitely generated. If $\Gamma$ is finitely generated, $\widehat{k[(u_i)_{i\in I}]}$ is the usual power series ring in finitely many variables, equipped with a weight.}
 \begin{theorem}\label{cohen}{{\rm (The valuative Cohen theorem, part 1)}} Let $R$ be a complete equicharacteristic noetherian local domain and let $\nu$ be a rational valuation of $R$. We fix a field of representatives $k\subset R$ of the residue field $R/m$ and a minimal system of homogeneous generators $(\overline\xi_i)_{i\in I}$ of the graded $k$-algebra ${\rm gr}_\nu R$. There exist choices of representatives $\xi_i\in R$ of the $\overline\xi_i$ such that the application $u_i\mapsto \xi_i$ determines a surjective map of $k$-algebras 
 $$\pi\colon \widehat{k[(u_i)_{i\in I}]}\longrightarrow R$$
 which is continuous with respect to the topologies associated to the filtrations by weight and by valuation respectively. The associated graded map with respect to these filtrations is the map
 $${\rm gr}_w\pi\colon k[(U_i)_{i\in I}]\longrightarrow {\rm gr}_\nu R,\ \ U_i\mapsto\overline\xi_i$$
whose kernel is a prime ideal generated by binomials $(U^{m^\ell}-\lambda_\ell U^{n^\ell})_{\ell\in L},\ \lambda_\ell\in k^*$.\par\noindent
If the semigroup $\Gamma=\nu(R\setminus\{0\})$ is finitely generated or if the valuation $\nu$ is of rank one, one may take any system of representatives $(\xi_i)_{i\in I}$. \end{theorem}
This is proved in \cite[\S 4]{Te3} and we recall that in view of \cite[Proposition 3.6]{Te3}, the initial form ${\rm in}_wh(u)$ of an element $h(u)\in \widehat{k[(u_i)_{i\in I}]}$ is in the kernel of ${\rm gr}_w\pi$ if and only if $w(h(u))<\nu (\pi(h(u)))$.
\begin{remark}\label{HS}1) After giving this result its name, I realized that Cohen's structure theorem had roots in valuation theory. It is an analogue for complete noetherian local rings of a structure theorem for complete valued fields due to H. Hasse and F.K. Schmidt. See \cite[Section 4.3]{Ro}. However, this result was itself inspired by a conjecture of Krull in commutative algebra; see \cite[Introduction]{Co}.\par\medskip\noindent
2) The sum $\sigma=\sum u^m$ of \textit{all} monomials is an element of $\widehat{k[(u_i)_{i\in I}]}$ and can be deemed to represent the transfinite product $\prod_{i\in I}(\frac{1}{1-u_i})$ since in the expansion of $\frac{1}{1-u_i}$ each power of $u_i$ appears exactly once and with coefficient $1$.\par\noindent If we consider the canonical map of topological $k$-algebras $\widehat{k[(u_i)_{i\in I}]}\to k[[t^\Gamma]]$ determined by $u_i\mapsto t^{\gamma_i}$ defined in \cite[introduction to section 4]{Te3}, by construction a monomial of weight $\gamma$ in the variables $u_i$ encodes a representation of $\gamma$ as a combination of the generators $\gamma_i$ so that the image of the series $\sigma$ is the generalized generating series $\sum_{\gamma\in\Gamma} p(\gamma) t^\gamma\in k[[t^\Gamma]]$ where $\gamma\mapsto p(\gamma)\in\N$ is a partition function: the number of distinct ways of writing $\gamma$ as a combination with coefficients in $\N$ of the generators $\gamma_i$. We see that it satisfies a generalized version of Euler's identity for the generating series of partitions: $$\sum_{\gamma\in\Gamma} p(\gamma) t^\gamma =\prod_{i\in I}(\frac{1}{1-t^{\gamma_i}}).$$\par\noindent  We observe that the semigroup $\Phi_{\geq 0}$ is itself well ordered and combinatorially finite if and only if its approximation sequence 
$$\N^r_0\subset \N^r_1\subset \cdots \subset \N^r_h\subset \N^r_{h+1}\subset\cdots\Phi_{\geq 0}$$by nested free subsemigoups $\N^r$, where $r$ is the rational rank of $\Phi$,  following from \cite[Proposition 2.3]{Te3} is finite. Indeed, if it is infinite there must be infinite decreasing sequences of elements of $\Phi_{\geq 0}$ because if the linear inclusion $\N_h^r\subset \N_{h+1}^r\subset \Phi_{\geq 0}$, encoded by a matrix with entries in $\N$, is not an equality, there has to be a basis element of $\N_{h+1}^r$ which is strictly smaller than the minimum of elements of $\N_h^r$. Conversely, there can be no infinite decreasing sequence of elements of $\N^r$ with a monomial ordering where all elements are $\geq 0$.\par Therefore, if $\Phi_{\geq 0}$ is well ordered, we have $\Phi_{\geq 0}=\N^r$ where $r$ is the rational rank of $\Phi$. The same construction as above can be applied when taking variables $u_i$ indexed by all the elements of $\N^r$ instead of only generators and taking $r=1$ we recover the usual partition function $p(n)$ and the usual Euler identity.\par\noindent 
The series $\sum_{\gamma\in\Gamma} p(\gamma) t^\gamma $ can also be viewed as the generalized Hilbert-Poincar\'e series of the graded algebra ${\rm gr}_w\widehat{k[(u_i)_{i\in I}]}$ since $p(\gamma)$ is the dimension of the $k$-vector space generated by the monomials of degree $\gamma$. We note that the Hilbert-Poincar\'e series of ${\rm gr}_\nu k[[t^\Gamma]]=k[t^\Gamma]$ is again, up to a change in the meaning of the letter $t$, the sum $\sum_{\gamma\in\Gamma}t^\gamma$ of all monomials.
\end{remark}
\end{subsection}
\begin{subsection}{The valuative Chevalley Theorem}
In this section we dig a little deeper into the relationship between the valuative Cohen theorem and Chevalley's theorem (see \cite{Te1}, section 5, and \cite{B}, Chap. IV, \S 2, No. 5, Cor.4), which is essential in its proof. Let $R$ be a complete equicharacteristic noetherian local domain and let $\nu$ be a rational valuation centered in $R$. Let us take as set $(\gamma_i)_{i\in I}$ the entire semigroup $\Gamma$. If the valuation $\nu$ is of rank one, this set is of ordinal $\omega$ (see \cite[Appendix 3, Proposition 2]{Z-S}) and cofinal in $\Phi_{\geq 0}$ since by \cite[Theorem 3.2 and \S 4]{C-T} the semigroup $\Gamma$ has no accumulation point in $\R$ because we assume that $R$ is noetherian. Since we have $\bigcap_{\gamma\in \Gamma}\Pp_{\gamma}=(0)$, it is an immediate consequence of Chevalley's theorem that given a sequence of elements of $R$ of strictly increasing valuations, their $m$-adic orders tend to infinity. This is no longer true for valuations of rank $>1$ as evidenced by the following:
\begin{example}\label{nochev} Let $R=k[[x,y]]$ equipped with the monomial rank two valuation $\mu$ with value group $(\Z^2)_{lex}$ such that $\mu(x)=(1,0)$ and $\mu(y)=(0,1)$. The elements $(x+y^i)_{i\geq 1}$ have strictly increasing values $(0,i)$ but their $m$-adic order remains equal to $1$. \end{example}
This is the reason why in the valuative Cohen theorem the representatives of generators of the associated graded algebra have to be chosen. In this section we present some consequences of this choice.\par\medskip
Let $(\gamma_i)_{i\in I}$ be a well ordered set contained, with the induced order, in the non negative semigroup of a totally ordered group $\Phi$ of finite rational rank. By the theorem of B. Neumann in \cite{N} which we have already quoted, the semigroup $\Gamma$ generated by those elements is well ordered. Let
$$(0)=\Psi_h\subset \Psi_{h-1}\subset\cdots\subset\Psi_1\subset\Psi_0=\Phi$$
be the sequence of convex subgroups of $\Phi$, where $h$ is the rank of $\Phi$ and $\Psi_k$ is of rank $h-k$. Let us consider the canonical map $\lambda\colon \Phi\rightarrow \Phi/\Psi_{h-1}$. 
With these notations we can state the following inductive definition:
\begin{definition}\label{definit} A subset $B$ of a well ordered subset $(\gamma_i)_{i\in I}$ of $\Phi_{\geq 0}$ is said to be \textit{initial in $(\gamma_i)_{i\in I}$} when
 \begin{itemize} 
\item If $h=1$, then $B$ is of the form $\{ \gamma_i\vert i\leq i_0\}$.
\item If $h>1$, then $\lambda(B)$ is initial in the set $(\lambda(\gamma_i))_{i\in I}\subset  \Phi/\Psi_{h-1}$ of the images of the $\gamma_i$ and for every $\phi_{h-1}\in \lambda(B)$ we have that the set of differences $B\cap\lambda^{-1}(\phi_{h-1})-\tilde\phi_{h-1}$ is initial in the set of differences $\{\gamma_i)_{i\in I}\}\cap\lambda^{-1}(\phi_{h-1})-\tilde\phi_{h-1}\subset (\Psi_{h-1})_{\geq 0}$, where $\tilde\phi_{h-1}$ is the smallest $\gamma_i$ contained in $\lambda^{-1}(\phi_{h-1})$.
\end{itemize}
\end{definition}
 One verifies by induction on the rank that the intersection of initial subsets is initial, so that we have an initial closure of a subset $C$ of $(\gamma_i)_{i\in I}$, the intersection of the initial subsets containing $C$.
 \begin{remark}\label{plusone} An initial subset $B$ has the property that if $\gamma_i\notin B$, then $\gamma_{i+1}\notin B$, where as usual $i+1$ is the successor of $i$ in $I$. The converse is false when the rank of the valuation is $>1$ and there are elements without predecessor in $(\gamma_i)_{i\in I}$. It is not true either, if the valuation is of rank $>1$, that if $\gamma_i\in B$ and $\gamma_j<\gamma_i$, then $\gamma_j\in B$ or that if $\gamma_j\notin B$ and $\gamma_k>\gamma_j$, then $\gamma_k\notin B$.\par The useful features of initial sets are what comes now and especially the valuative Chevalley Theorem below which shows that finite initial subsets provide a way of approximating the countable ordinal $I$ indexing the generators of $\Gamma$ by nested finite subsets which is appropriate for the valuative Cohen Theorem.
 \end{remark}
 
 \textit{To simplify notations, we shall sometimes identify each $\gamma_i$ with its index $i$, as in the proof below.}
\begin{proposition}\label{initial}Let the $(\gamma_i)_{i\in I}$ be as above. The initial closure of a finite subset of $(\gamma_i)_{i\in I}$ is finite.
\end{proposition}
\begin{proof} This is a variant of the proof of lemma 5.58 of \cite{Te1}. Let $C$ be a finite subset of $I$. If the valuation is of rank one the ordinal of $I$ is at most $\omega$ (see \cite[Appendix 3, Proposition 2]{Z-S}) and the initial closure of $C$ is the set of elements of $I$ which are less than or equal to the largest element of $C$, and it is finite. Assume now that the result is true for valuations of rank $\leq h-1$ and let $\nu$ be a valuation of rank $h$. Let $\lambda\colon \Phi\to\Phi_{h-1}=\Phi/\Psi_{h-1}$ be the map of groups
corresponding to the valuation $\nu_{h-1}$ of height $h-1$ with which $\nu$ is composed.  Set $C_1=\lambda (C)$ and let $I_1$ index the distinct $(\lambda(\gamma_i))_{i\in I}$. We have a natural monotone map $I\to I_1$, which we still denote by $\lambda$. By the induction hypothesis, we have a finite set $\tilde C_1$ containing $C_1$ and initial in  $I_1$. Define $\tilde C$ as follows: it is the union of finite subsets $\tilde C_{i_1}$ of the $\lambda^{-1}(i_1)$ for $i_1\in \tilde C_1$. Define $\tilde C_{i_1}$ to be the set of elements of $\lambda^{-1}(i_1)$ which are smaller than or equal to the largest element of $C\cap\lambda^{-1}(i_1)$. By \textit{loc. cit.},
Lemma 4, or Proposition 3.17 of \cite{Te1}, each of these sets if finite and $\tilde C=\bigcup_{i_1\in \tilde C_1}\tilde C_{i_1}$ is the initial closure of $C$. \end{proof}
\begin{corollary}\label{approx} Given a finite set $B_0\subset I$, there is a countable sequence $(B_t)_{t\in \N_{>0}}$ of nested finite initial subsets containing $B_0$ whose union is $I$:
$$B_0\subset B_1\subset \cdots\subset B_t\subset B_{t+1}\subset \cdots\subset I$$
\end{corollary}
\begin{proof} Consider the sequence of morphisms
$$\Phi\stackrel{\lambda_{h-1}}\longrightarrow \Phi/\Psi_{h-1}\stackrel{\lambda_{h-2}}\longrightarrow \Phi/\Psi_{h-2}\cdots\cdots\stackrel{\lambda_{2}}\longrightarrow \Phi/\Psi_2\stackrel{\lambda_{1}}\longrightarrow \Phi/\Psi_1\longrightarrow 0.$$
Denote by $\mu_i$ the compositum $\lambda_i\circ\lambda_{i-1}\circ\cdots\circ\lambda_{h-1}\colon\Phi\to\Phi/\Psi_i.$ Let $B_{0,1}\subset \Phi/\Psi_1$ be the initial closure of the union of the image of $B_0$ in $\Phi/\Psi_1$ and the smallest element of the image of $I$ in $\Phi/\Psi_1$ which is not in that image. For each $\phi_1\in B_{01}$ take the initial closure in $\lambda_{1}^{-1}(\phi_1)$ of the union of \break$\lambda_{1}^{-1}(\phi_1)\cap\mu_2(B_0)$ and the smallest element of $\lambda_{1}^{-1}(\phi_1)\cap\mu_2(I)$ which is not in $\lambda_{1}^{-1}(\phi_1)\cap\mu_2(B_0)$. If $\phi_1\notin\mu_1(B_0)$ take the smallest element of $\lambda_{1}^{-1}(\phi_1)\cap\mu_2(I)$. The union of the results of this construction is a finite initial subset $B_{02}$ of $\mu_2(I)$ which contains $\mu_2(B_0)$. \par\noindent
We repeat the same construction starting from $B_{02}\subset\mu_2(I)$ to build a finite initial subset $B_{03}$ of $\mu_3(I)$, and so on until we have created a finite initial subset $B_1$ of $I$ which strictly contains $B_0$ unless $B_0=I$. Then we apply the same construction replacing $B_0$ by $B_1$ to obtain $B_2$, and so on.\par
  Let us now prove that the union of the $B_t$ is $I$. We keep the notations of Definition \ref{definit} and Proposition \ref{initial}, where $\lambda_{h-1}=\lambda$. The result is true if $\Phi$ is of rank one because a strictly increasing sequence of elements of $I$ is cofinal in $I$ (see  \cite[Theorem 3.2]{C-T}). Let us assume that it is true for groups of rank $\leq h-1$ and assume that the result is not true for rank $h$. Let $\iota$ be the smallest element of $I$ which is not in $\bigcup_{t\in\N} B_t$. Using our induction hypothesis, define $t_0$ as the least $t$ such that $\lambda(\iota)\in \lambda (B_t)$. Let $\tilde\phi_\iota$ be the smallest element of $I\cap\lambda^{-1}(\lambda(\iota))$. By the definition of initial closure, for $t\geq t_0$ the elements of each set of differences $B_t\cap\lambda^{-1}(\lambda(\iota))-\tilde\phi_\iota$ are initial in $I\cap \lambda^{-1}(\lambda(\iota))-\tilde\phi_\iota\subset \Psi_{h-1}$. We are reduced to the rank one case and since by construction the $B_t\cap\lambda^{-1}(\lambda(\iota))$ grow with $t$ we obtain that for large enough $t$ we have $\iota\in B_t\cap\lambda^{-1}(\lambda(\iota))$ and a contradiction.
\end{proof}
\begin{theorem}\label{Chevalley}{{\rm (The valuative Chevalley theorem)}} Let $R$ be a complete equicharacteristic noetherian local domain and let $\nu$ be a rational valuation of $R$. Denote by $\Gamma$ the semigroup of values of $\nu$ and by $(\gamma_i)_{i\in I}$ a minimal set of generators of $\Gamma$. Assume that the set $I$ is infinite and let $D$ be an integer. Let $B_0$ be a finite subset of $I$. There exist elements $(\xi_i)_{i\in I}$ in $R$ such that $\nu(\xi_i)=\gamma_i$ and a finite initial subset $C(D)$ of $(\gamma_i)_{i\in I}$ containing $B_0$ and such that whenever $\gamma_i\notin C(D)$, then $\xi_i\in m^{D+1}$. 
 \end{theorem}
\begin{proof}This is essentially a consequence of the valuative Cohen theorem of \cite{Te3}, or rather of its proof. Recall that the choice of the $\xi_i$ in the proof of the valuative Cohen theorem is such that if the image of $\xi_i$ in a quotient $R/p_q$ is in a power of the maximal ideal, then $\xi_i$ lies in the same power of the maximal ideal in $R$, for all primes $p_q$ which are the centers of valuations with which $\nu$ is composed. So if $p_1$ is the center of the valuation $\nu_1$ of height one with which $\nu$ is composed, we may assume by induction that the result is true for the $\xi_i$ whose value is in the convex subgroup $\Psi_1$ of $\Phi$ associated to $\nu_1$. By a result of Zariski (see \cite[corollary 5.9]{Te1}), we know that the $\nu_1$-adic topology on $R$ is finer than the $m$-adic topology, which means that there exists a $\phi_1\in\Phi_1=\Phi/\Psi_1$ such that if $\nu_1(\xi_i)\geq \phi_1$ then $\xi_i\in m^{D+1}$. So we are interested only in those $\xi_i$ whose $\nu_1$-value is positive and less that $\phi_1$. If $\nu_1(m)>0$, they are finite in number (see \cite[proposition 3.17]{Te1}) so it suffices to add them to the finite set constructed at the previous inductive stage. If not, since $\Phi_1$ is of rank one, there are only finitely many values of the $\nu_1(\xi_i)$ which are positive and less than $\phi_1$. Let  $\eta$ be one of these. Denote by $\lambda\colon\Phi\to\Phi_1$ the natural projection. Again using the argument of the proof of the valuative Cohen theorem, and in particular Chevalley's theorem applied to the filtration of $\Pp_\eta/\Pp^+_\eta$ by the $\Pp_\phi/\Pp^+_\eta$ with $\lambda(\phi)=\eta$, we have that for each $\eta$ there is a $\phi_0(\eta)\in \lambda^{-1}(\eta)$ such that if $\nu(\xi_i)\in\lambda^{-1}(\eta)$ and  $\nu(\xi_i)\geq\phi_0(\eta)$ then $\xi_i\in m^{D+1}$. Again by (\cite{Te1}, proposition 3.17) there are finitely many generators of $\Gamma$ in $\lambda^{-1}(\eta)$ which are $\leq \phi_0(\eta)$. It suffices to add to our finite set the union of these finite sets over the finitely many values $\eta$ that are $\leq \phi_1$. 
 \end{proof}
 \begin{proposition} If the rational valuation $\nu$ is of rank one, Theorem \ref{Chevalley} is equivalent to the existence for each integer $D>0$ of a finite initial set $C(D)$ in $\Gamma$ such that for $\gamma\notin C(D)$ we have $\Pp_\gamma\subset m^{D+1}$.
 \end{proposition}
 \begin{proof} If the set $I$ is infinite, the $\gamma_i$ are cofinal in $\Phi_{\geq 0}$ since by \cite[Theorem 3.2 and \S 4]{C-T} the semigroup $\Gamma$ has no accumulation point in $\R$ because we assume that $R$ is noetherian. Therefore we have $\bigcap_{i\in I}\Pp_{\gamma_i}=(0)$ where the ordinal of $I$ is $\omega$. By Chevalley's theorem (see \cite[Chap. IV, \S 2, No. 5, Cor. 4]{B}), all but finitely many of the $\xi_i$ are in $m^{D+1}$. Conversely, assuming the result of Theorem \ref{Chevalley}, recall that $\Pp_\gamma$ is generated by the monomials $\xi^e$ of value $\geq \gamma$ and let $s\in I$ be the smallest index such that $t\geq s$ implies that $\xi_t\in m^{D+1}$. Let $(\gamma_i)_{i\in J}$ be the finite set of those $\gamma_j<\gamma_s$. For any monomial $\xi^e$ such that some decomposition of $\nu(\xi^e)$ as sum of $\gamma_i$ contains a $\gamma_t$ with $t\geq s$ we have $\xi^e\in m^{D+1}$. Let $C(D)$ be the set of elements of $\Gamma$ which are $\leq D\gamma_s$. If $\nu(\xi^e)=\gamma >D\gamma_s$, either its decomposition along the $\gamma_i$ contains a $\gamma_t$ with $t\geq s$ or it can be written $\gamma=\sum_{i\in J}a_i\gamma_i$ with $a_i\in \N$ and then since $(\sum_{i\in J}a_i){\rm sup}(\gamma_i)_{i\in J}> D\gamma_s$ we must have $(\sum_{i\in J}a_i)\geq D+1$ and thus $\xi^e\in m^{D+1}$ since the $\xi_i$ are in $m$.\par\noindent
 If the set $I$ is finite, taking as initial set $C(D)$ the set of elements of $\Gamma$ which are $\leq D\gamma_s$, where now $\gamma_s$ is the largest of the $\gamma_i$, the last argument of the proof shows that $\Pp_\gamma\subset m^{D+1}$ for $\gamma\notin C(D)$. \end{proof}
 \begin{remark} 
So we see that in the rank one case Theorem \ref{Chevalley} is indeed equivalent to the result given by Chevalley's theorem.  \end{remark}
 \begin{corollary}\label{chev} In the situation of the valuative Cohen theorem, given an infinite collection of monomials $(u^{q_\ell})_{\ell\in L }$ such that each monomial appears only finitely many times, for any integer $D$ all but finitely many of the $\xi^{q_\ell}$ are in $m^{D+1}$. 
\end{corollary}
\begin{proof} Set $C=C(D)$ and let us consider those monomials $u^{q_\ell}$ which involve only the finitely many variables $u_i$ with $i\in C$. If there are finitely many such monomials, we are done. Otherwise, since each monomial appears only finitely many times, and we are now dealing with finitely many variables, there are only finitely many of those monomials which are such that $\vert q_\ell\vert\leq D$ and we are done again since the $\xi_i\in m$ and by construction all the other monomials are in $m^{D+1}$.
\end{proof}
\begin{remark}Since each monomial appears only finitely many times, the sum $\sum_{\ell\in L}u^{q_\ell}$ exists in $\widehat{k[(u_i)_{i\in I}]}$ and therefore its image $\sum_{\ell\in L}\xi^{q_\ell}$ must converge in $R$, albeit in a transfinite sense. The corollary shows an aspect of this.  \end{remark}
\end{subsection}
\begin{subsection}{Finiteness in the valuative Cohen theorem}\label{finiteness} When the rational valuation $\nu$ on the noetherian complete local domain $R$ is of rank $>1$ and the semigroup $\Gamma$ is not finitely generated, not only does one have to choose the representatives $\xi_i\in R$ of generators of the $k$-algebra ${\rm gr}_\nu R$ carefully (see \cite{Te3}) in order to avoid adding infinitely many times the same element in sums such as $\sum_{i\in I}\xi_i$, but the equations $F_\ell$ whose initial forms are the binomials $(u^{m^\ell}-\lambda_\ell u^{n^\ell})_{\ell\in L},\ \lambda_\ell\in k^*$ also have to be chosen carefully in order to avoid similar accidents when writing elements of the closure in $\widehat{k[(u_i)_{i\in I}]}$ of the ideal which they generate. In this subsection we show how the noetherianity of the local domain $R$ can be used to prove the existence of good choices of equations.
Since the binomials $(U^{m^\ell}-\lambda_\ell U^{n^\ell})_{\ell\in L},\ \lambda_\ell\in k^*$ generate the $w$-initial ideal of the kernel $F$ of the map $\pi\colon \widehat{k[(u_i)_{i\in I}]}\longrightarrow R$, for each $\ell\in L$ there is at least one element of the form $u^{m^\ell}-\lambda_\ell u^{n^\ell}+\sum_{w(p)>w(u^{m^\ell})}c_pu^p$ which is in $F$ (overweight deformation of its initial binomial). Let us call such elements $F_\ell$.
\begin{proposition}\label{finite} The elements $F_\ell$ can be chosen, without modifying the initial binomial, in such a way that each one involves only a finite number of variables and no monomial $u^p$ appears in infinitely many of them or in infinitely many products $A_\ell F_\ell$, where the $A_\ell$ are terms such that there are at most finitely many products $A_\ell F_\ell$ of any given weight.\par\noindent
 \end{proposition}
\begin{proof} According to equation $(E)$ of subsection \ref{reminder}, the image in $R$ of a topological generator is of the form $\xi^{m^\ell}-\lambda_\ell\xi^{n^\ell}- \sum_pc^{(\ell)}_p\xi^p$ with $\nu(\xi^p)>\nu(\xi^m)$ for all exponents $p$ and of course it is zero in $R$. Since the ring $R$ is noetherian, the ideal generated by the $\xi^p$ which appear in the series is finitely generated, say by $\xi^{e_1},\ldots ,\xi^{e_s}$. Let us choose finitely many of the $\xi_j$ which generate the maximal ideal of $R$, call them collectively $\Xi$ and call $U$ the collection of the corresponding variables in $\widehat{k[(u_i)_{i\in I}]}$ . Then our series can be rewritten as $\xi^{m^\ell}-\lambda_\ell\xi^{n^\ell}- \sum_{j=1}^s G^{(\ell)}_j(\Xi)\xi^{e_j}$, where $ G^{(\ell)}_j(\Xi)\in R$ is a series of monomials in the elements of $\Xi$. Our series is the image of the element \break$u^{m^\ell}-\lambda_\ell u^{n^\ell}- \sum_{j=1}^sG^{(\ell)}_j(U)u^{e_j}$ of $F$.  This element has the same initial binomial as our original series since $w(u^{e_j})>w(u^{m^\ell})$ for all $j$ and involves only finitely many variables. It can replace our original series as element of the ideal $F$ with initial form $u^{m^\ell}-\lambda_\ell u^{n^\ell}$.\par
To prove the second assertion, remember from \cite[\S 4]{Te3} that there is a well-ordering on the set of all monomials obtained by embedding it into $(\N\times\Gamma)_{lex}$ thanks to the finiteness of the fibers of the map $u^e\mapsto w(u^e)$. Fix a choice of $(F_\ell)_{\ell\in L}$ and consider all $k$-linear combinations of the $F_\ell$ of the form $F_\ell -\mu F_{\ell'}$ which do not modify the initial binomial. If the set of monomials $u^e$ which appear infinitely many times as terms in all such linear combinations is not empty, it has a smallest element $u^{e_0}$. Note that we may assume that such a monomial does not appear in the initial binomials because there are only finitely many initial binomials involving a given finite set of variables.\par\noindent Denote by $L_0$ the set of indices $\ell\in L$ such that $u^{e_0}$ appears in $F_\ell$ as above, and  by $L_1\subset L_0$ the finite set of indices $\ell\in L_0$ such that the weight of the initial form of $F_\ell$ is minimal. Denote by $L_2$ the set of indices in $L_0\setminus L_1$ such that the weight of the initial form of $F_\ell$ is minimal, and define recursively in the same manner finite subsets $L_\iota$ indexed by an ordinal.  Now for each $\ell\in L_1$ we can replace $F_\ell$ by $F_\ell-\mu_{\ell,\ell'}F_{\ell'}$ with an $\ell'\in L_2$ without changing the initial binomial, where the constant $\mu_{\ell,\ell'}$ is chosen to eliminate the monomial $u^{e_0}$ from the difference. In so doing we create a new system of overweight deformations of the collection of binomial ideals and cannot add a monomial of lesser weight than $u^{e_0}$ which might appear infinitely many times since $u^{e_0}$ was the smallest in the original family $(F_\ell)_{\ell \in L}$. We continue this operation with $\ell\in L_\iota,\ell'\in L_{\iota+1}$. In the end we have a new system of overweight deformations obtained by linear combinations and in which $u^{e_0}$ cannot appear. This contradiction shows that we can choose the $F_\ell$ so that no monomial appears infinitely many times. \par\noindent It follows from lemma \ref{appear} that under our assumption, for any collection of isobaric polynomials $A_\ell$, the $A_\ell (F_\ell -(u^{m^\ell}-\lambda_\ell u^{n^\ell}))$, cannot contain infinitely many times the same monomial. On the other hand since each $A_\ell$ is an isobaric polynomial, a monomial can appear only finitely many times in all the $A_\ell u^{m^\ell}$ or $A_\ell u^{n^\ell}$, so that altogether no monomial can appear infinitely many times in the collection of series $A_\ell F_\ell$.
\end{proof}
\begin{theorem}\label{gener}{{\rm (The valuative Cohen theorem, part 2)}}
The kernel $F$ of $\pi$ is the closure of the ideal generated by the elements $F_\ell$ obtained as in Proposition \ref{finite} as $U^{m^\ell}-\lambda_\ell U^{n^\ell}$ runs through a set of generators of the kernel of ${\rm gr}_w\pi$. In a slightly generalized sense, these generators form a standard basis of the kernel of $\pi$ with respect to the weight $w$.
\end{theorem}
\begin{proof} Let $h$ be a non zero element of $F$; its weight is finite but its image in $R$ has infinite value. Therefore, ${\rm in}_wh$ belongs to the ideal $F_0$ generated by the $(U^{m^\ell}-\lambda_\ell U^{n^\ell})_{\ell\in L}$.\par\noindent Thus, there exists a finite set of elements $(F_\ell)_{\ell\in L_1}$, with $ L_1\subset L$, and isobaric polynomials $(A_\ell)_{\ell\in L_1}$ in the $u_i$ such that $w(h-\sum_{\ell \in L_1}A^{(1)}_\ell F_\ell)>w(h)$. Since this difference is still in $F$, we can iterate the process and build a series indexed by the set $T$ of weights of the successive elements, say $y_\tau$, with $y_0=h, y_1=h-\sum_{\ell \in L_1}A^{(1)}_\ell F_\ell$ and, if $\tau+1$ is the successor of $\tau$ in the index set $T$, $y_\tau-y_{\tau+1}=\sum_{\ell \in L_\tau}A^{(\tau)}_\ell F_\ell$, such that $w(y_{\tau+1})>w(y_\tau)$, so that $w(y_\tau)= w(y_{\tau+1}-y_\tau)$. The sequence $(y_\tau)_{\tau\in T}$  is a pseudo-convergent sequence with respect to the ultrametric $u(x,y)=w(y-x)$ which, according to \cite[ \S 4]{Te3} has $0$ as a pseudo-limit in the spherically complete ring $ \widehat{k[(u_i)_{i\in I}]}$. This is of course not sufficient to prove what we want. However,  in view of proposition \ref{finite}, no monomial can appear in infinitely many of the $y_\tau-y_{\tau+1}$. If $h\neq \sum_{\tau<\rho}(y_{\tau+1}-y_\tau)$, the initial form of $h- \sum_{\tau<\rho}(y_{\tau+1}-y_\tau)$ is in $F_0$ and we can continue the approximation. Thus, the transfinite sum $\sum_{\tau\in T}(y_\tau-y_{\tau+1})$ exists in $\widehat{k[(u_i)_{i\in I}]}$ and, by definition of $T$, has to be equal to $h$ so that $0$ is indeed the limit of the sequence $y_\tau$. This shows that $h$ is in the closure of the ideal generated by the $F_\ell$. \end{proof}
\begin{remark}\label{error}The statement about the $F_\ell$ topologically generating the kernel of $\pi$ is part a) of the asterisked\footnote{An asterisked propositions in that text means it is endowed with hope but not with a proof.}  proposition 5.53 of \cite{Te1}. Part b) of that proposition is incorrect, as the next section shows.
\end{remark}
\textit{From now on we shall assume that each of the equations $F_\ell$ we consider involves only finitely many variables and the $F_\ell$'s topologically generate the kernel of $\pi$}.\par\medskip\noindent

Coming back to the valuative Cohen theorem itself, if the valuation $\nu$ is of rank one or if the semigroup $\Gamma$ is finitely generated, the result will be true for any choice of representatives $\xi_i$ (see \cite{Te3}, Remarks 4.14).\par In view of their definition the $(\xi_i)_{i\in I}$ generate the maximal ideal of $R$. Since $R$ is noetherian there is a finite subset $J\subset I$ such that the elements $(\xi_i)_{i\in J}$ minimally generate this maximal ideal. According to (\cite{Te1}, 5.5), for each $\ell\in I\setminus  J$ there must be among the topological generators $(E)$ of the ideal $F$ one which contains $u_\ell$ as a term. We detail the argument here:\par\noindent The element $\xi_\ell$ is expressible as a series $h((\xi_i)_{i\in J})$. Therefore the series $u_\ell-h((u_i)_{i\in J})$ must belong to the ideal $F$.  There must be an element $H$ of the closure of the ideal generated by the series $(u^{m^\ell}-\lambda_\ell u^{n^\ell}+\sum_{w(p)>w(m^\ell)}c_p^{(\ell)}u^p)_{\ell\in L}$ such that $w(u_\ell-h((u_i)_{i\in J})-H)>\gamma_\ell$. Thus, the series $H$ must contain $u_\ell$ as a term so that at least one of the topological generators of $F$ must contain $u_\ell$ as a term. An argument given in (\cite{Te3}, proof of proposition 7.9) suggests that:
\begin{proposition}\label{lin} Up to a change of the representatives $\xi_i$ we may assume that each variable $u_i$ with $i\in I\setminus J$ appears as a term in a topological generator of ${\rm ker}\pi$ of the form:
$$u^{m^i}-\lambda_i u^{n^i}-g_i((u_q)_{q<i})-u_i,\eqno{(F_i)}$$ where every term of the series $g_i((u_q)_{q<i})$ has weight $>w(u^{m^i})$ and $<w(u_i)$.
\end{proposition}
\begin{proof} Up to multiplication of variables by a constant and corresponding modification of the $\lambda_\ell$, we may assume that for each $i\in I\setminus J$, the variable $u_i$ appears with coefficient minus one in $F_i$. Let $i_1$ be the smallest element of $I\setminus J$ such that there exists an equation $F_i$ as above containing $u_i$ linearly and not having the form described above. We write $g((u_q)_{q<i_1})$ for the series of terms of weight $<i_1$. The remaining part of $F_{i_1}$ can be written $-u_{i_1}+\sum_{w(p)>\gamma_{i_1}}c^{(i_1)}_pu^p$. If we replace the representative $\xi_{i_1}\in R$ by $\xi_{i_1}-\sum_{w(p)>\gamma_{i_1}}c^{(i_1)}_p\xi^p$ without changing the previous $\xi_j$, we modify the equation $F_{i_1}$ into the desired form. We note that in view of the definition of $i_1$ we may assume that no variable of index $<i_1$ appears as a term in $(F_{i_1})$. We then continue with the smallest index $i_2$ such that the corresponding equation does not have the desired form, and obtain the result by transfinite induction.
\end{proof}

\begin{proposition}\label{irr} Each $F_\ell$ is irreducible in $\widehat{k[(u_i)_{i\in I}]}$.
\end{proposition}
\begin{proof} Should it be reducible, since the kernel $F$ of $\pi$ is prime, one of its factors should be in $F$ and by Theorem \ref{gener} and the fact that the kernel of ${\rm gr}_w\pi$ is prime, the initial binomial of $F_\ell$ should be in the ideal generated by the initial forms of the other generators of $F$, which is impossible in view of the assumption we made on the generating binomials.\par
\end{proof}
\begin{remark}
Since the initial forms of the $F_\ell$ are the binomials defining our toric variety ${\rm Spec}{\rm gr}_\nu R$, Theorem \ref{gener} indeed provides equations for the degeneration to ${\rm Spec}{\rm gr}_\nu R$. However, since the binomials $(U^{m^\ell}-\lambda_\ell U^{n^\ell})_{\ell\in L}$ come in bulk, it is difficult to make use of them. The next two subsections address this difficulty.
\end{remark}
\end{subsection}

\begin{subsection}{The valuative Cohen theorem for power series rings}\label{cohenreg}
In this section we prove that, in the case where $R$ is a power series ring, the equations provided by the valuative Cohen theorem can then be given a more specific form.\par\noindent
Assume that $R$ is a power series ring and we have chosen representatives $(\xi_i)_{i\in I}$ with which we can apply the valuative Cohen theorem. Since the $\xi_i$ generate the maximal ideal, we can find a subset $J\subset I$ such that the $(\xi_i)_{i\in J}$ form a minimal system of generators of the maximal ideal of $R$. Note that we state nothing about the rational independence of their valuations. 
By Proposition \ref{lin} for each $i\in I\setminus J$ there is a series of the form 
$$u^{m^i}-\lambda_i u^{n^i}-g_i((u_q)_{q<i})-u_i,\eqno{(F_i)}$$
among the topological generators of ${\rm ker}\pi$.\par\noindent Generators of the relations between the values of the $(\xi_i)_{i\in J}$ are encoded by finitely many binomials $(U^{m_\ell}-\lambda_\ell U^{n_\ell})_{\ell \in L_0}$ which according to Theorem \ref{gener} are the initial forms for the weight $w$ of series $(F_\ell)_{\ell \in L_0}$ which are among the topological generators of ${\rm ker}\pi$. 
\begin{theorem}\label{regular} If $R$ is regular, i.e., a power series ring with coefficients in $k$, the kernel of the map $\pi\colon \widehat{k[(u_i)_{i\in I}]}\to R$ asssociated to a rational valuation $\nu$ by Theorems \ref{cohen} and \ref{gener} is the closure of the ideal generated by the $(F_i)_{i\in I\setminus J}$.
\end{theorem}
\begin{proof} It suffices to show that all the $(F_\ell)_{\ell \in L}$ of Theorem \ref{gener} are in the ideal generated by the $F_i$. It is impossible for any $F_\ell$ to use only variables with indices in $J$ because if that were the case its image in $R$ by $\pi$ would be a non trivial relation between the $(\xi_i)_{i\in J}$, contradicting the fact that these are a minimal set of generators of the maximal ideal in a regular local ring.\par\noindent Given $F_\ell$ with $\ell\in L$, let $i$ be the largest index which is not in $J$ of a variable appearing in $F_\ell$. Let us substitute $F_i+u_i$ in place of $u_i$ in $F_\ell$, according to corollary \ref{substitute}.  Let us denote by $\tilde F_\ell$ the result of this substitution. It involves only finitely many variables, and those whose indices are not in $J$ are all of index $<i$. The difference $F_\ell -\tilde F_\ell$ is a multiple of $F_i$. If $\tilde F_\ell=0$, since $F_\ell$ is irreducible it should be equal to $F_i$ up to multiplication by a unit, which proves the assertion in this case. Thus we may assume that $\tilde F_\ell\neq 0$. Since the initial form ${\rm in}_w\tilde F_\ell$ can have only finitely many factors, the same is true for $\tilde F_\ell$ and since the kernel of $\pi$ is prime, at least one of these factors belongs to it. The variables appearing in this factor are a subset of those appearing in $\tilde F_\ell$. Now we can again take the variable of highest index not in $J$, say $k$, appearing in this factor, make the substitution using the corresponding $F_k$ and repeat the argument; if the result of the substitution is zero, then $F_\ell$ belonged to the ideal generated by $F_i$ and $F_k$. Otherwise we have an irreducible element in ${\rm ker}\pi$ involving only variables of index $<k$ and variables with index in $J$. Since the indices are elements of a well ordered set, after finitely many such steps, either we obtain zero which shows what we want, or we obtain a non zero element of ${\rm ker}\pi$ involving only the variables $(u_i)_{i\in J}$. Again the image of this element by $\pi$ would be a non trivial relation between the elements of a minimal system of generators of the maximal ideal of $R$, which is impossible since $R$ is regular. Thus we obtain a contradiction if $F_\ell$ did not belong to the ideal generated by the $F_i$. 
\end{proof}
\begin{remark}\label{choice} 

1) By the faithful flatness of the toric degeneration, each binomial $u^{m^\ell}-\lambda_\ell u^{n^\ell}$ is the initial binomial of an equation $F_\ell$ and thus the $F_\ell$ are a (generalized) standard basis for the kernel of $\pi$ (see \cite[Proposition 5.53]{Te1}). We do not know whether the $F_i$, which topologically generate the same ideal, are such a standard basis.\par\noindent 
2) Theorem \ref{regular} can be interpreted as a manifestation of the "abyssal phenomenon" of \cite[section 5.6]{Te1}. If the values of the generators $(\xi_i)_{i\in J}$ of the maximal ideal of our power series ring are rationally independent, the semigroup $\Gamma$ is finitely generated, and we have an Abhyankar, even monomial, valuation. If such is not the case, the rational rank of the valuation is strictly less than the dimension of $R$ and we are in the non-Abhyankar case. The semigroup $\Gamma$ is not finitely generated and the binomial relations in ${\rm gr}_\nu R$ expressing the rational dependance of values of the $(\xi_i)_{i\in J}$ cannot give rise to relations in $R$, as they should in view of the faithful flatness of the degeneration of $R$ to ${\rm gr}_\nu R$, because $R$ is regular. The role of the equations $(F_i=0)_{i\in I\setminus J}$ is to prevent their expression in $R$, which would decrease the dimension and introduce singularities, by sending this expression to infinity. Of course one could also point out that something like this is necessary to deform a non-noetherian ring such as ${\rm gr}_\nu R$ into the noetherian $R$, but that leaves out the precise structure of the deformation, which is significant even when ${\rm gr}_\nu R$ is noetherian; see Example \ref{ex1}. \par\noindent
This indicates that it is probably not possible to approximate directly our non-Abhyankar valuation by Abhyankar valuations in this context. However, as we shall see in the sequel, using appropriate truncations of the system of equations $F_\ell$, one can produce a sequence of Abhyankar semivaluations on $R$ (Abhyankar valuations on quotients of $R$) which approximate our valuation $\nu$.\par\noindent
 \end{remark}
\end{subsection}
\begin{subsection}{Lifting the valuation to a power series rings}\label{CohenRegular}
We can present our complete equicharacteristic local domain $R$ as a quotient of a power series ring $S=k[[x_1,\ldots ,x_n]]$ over $k$ by a prime ideal $P$. Our rational valuation $\nu$ then appears as a valuation on $S/P$ which we can compose with the $PS_P$-adic valuation $\mu_1$ of the regular local ring $S_P$ to obtain a rational valuation $\mu$ on $S$. By definition of the symbolic powers $P^{(e)}=P^eS_P\cap S$, the associated graded ring of $S$ with respect to $\mu_1$ is ${\rm gr}_{\mu_1}S=\bigoplus_{e\in \N}\frac{P^{(e)}}{P^{(e+1)}}$.\par
The graded ring ${\rm gr}_\mu S$ is then the graded ring associated to the filtration of ${\rm gr}_{\mu_1}S$  induced on each homogeneous component by the valuation ideals $\Pp_\phi/P^{(e+1)}$ with $P^{(e+1)}\subseteq \Pp_\phi\subseteq P^{(e)}$, where $\phi\in \{e\}\oplus\Phi$.\par  Let us pick a minimal system of generators $p_1,\ldots, p_s$ for the ideal $P$. We note that a regular system of parameters for the regular local ring $S_P$ consists of $n-d$ of the $p_q$'s. We shall assume in the sequel that the dimension of $S$ is minimal, which means that its dimension is the embedding dimension of $R$.
\begin{proposition}\label{compose} 1) The valuation $\mu$ on $S$ obtained as explained above is rational and its group is $(\Z\oplus\Phi)_{lex}$.  \par\noindent 2) The generators $p_1,\ldots ,p_s$ of the ideal $P$ can be chosen so that $\mu(p_1)<\mu(p_2)<\ldots <\mu(p_s)$ and their $\mu$-initial forms ${\rm in}_\mu p_1,\ldots ,{\rm in}_\mu p_s$ are part of a minimal system of generators of ${\rm gr}_\mu S$. All the generators of the semigroup of $\mu$ which are of $\mu_1$-value $>0$ and are not the $\mu (p_j)$ are of value $>\mu (p_s)$.\par\noindent 
\par\noindent 3)
With such a choice the ideal $PS_P$ is generated by $p_1,\ldots ,p_{n-d}$ and in particular  the $\mu_1$-initial forms ${\rm in}_{\mu_1}p_1,\ldots,{\rm in}_{\mu_1}p_{n-d}$ of $p_1,\ldots ,p_{n-d}$ are non zero in $P/P^{(2)}$. \end{proposition}
\begin{proof} Let $S_\mu\subset S_{\mu_1}$ be the inclusion of valuation rings corresponding to the fact that $\mu$ is composed with $\mu_1$ and let $m_{\mu_1}$ be the intersection with $S_\mu$ of the maximal ideal of $S_{\mu_1}$. The valuation ring $R_\nu$ is equal to the quotient $S_\mu/m_{\mu_1}$ and has the same residue field as $S_\mu$.  By (\cite{V}, \S 4), the value group of $\mu$ is an ordered extension of $\Z$ by $\Phi$, which can only be $(\Z\oplus\Phi)_{lex}$.\par If $\mu(p_1)=\mu(p_2)$ for example, their initial forms in ${\rm gr}_\mu S$ are proportional since $\mu$ is rational, and we may replace $p_2$ by $p_2-\lambda p_1$ whose value is $>\mu (p_1)$ if $\lambda$ is chosen appropriately in $k$, and so on. 
In view of the example given on page 200 of \cite{E-M}, some of the series $p_i$ may belong to $P^{(2)}$. However, that is not possible for those which constitute a regular system of parameters for $PS_P$. \footnote{When the field $k$ is of characteristic zero, it is likely that none of the $p_i$ may belong to $P^{(2)}$. We do not know whether $P^{(2)}\subset (x_1,\ldots ,x_n)P$ (see \cite{E-M}) which would give the result immediately by Nakayama's Lemma, but we can obtain a partial result:\par\noindent If $p_j\in P^{(2)}$, there is an $h\notin P$ such that $hp_j\in P^2$. In characteristic zero this implies that $\frac{\partial p_j}{\partial x_i}\in P$ for $i=1,\ldots ,n$. By \cite[Chap. 0, 0.5]{Te}, $p_j$ is then integral over $mP$ so that if we denote by $P_1\subset P$ the ideal generated by those generators of $P$ which are not in $P^{(2)}$, we have $P_1+\overline{mP}=P$. Since the ideal $P$ is integrally closed, by the integral Nakayama Lemma of \cite[lemme 2.4]{Te}, with ${\mathfrak n}_1=0$, we have $\overline{P_1}=P$.}
\par To prove the second assertion of 2), let $\zeta\in S$ be a representative of the smallest element of the semigroup of $\mu$ on $S$ whose $\mu_1$-value is $>0$. This element has to be part of a minimal system of generators of ${\rm gr}_\mu S$. By construction we have $\mu(\zeta)\leq \mu(p_1)$. We also have  ${\rm in}_{\mu_1}\zeta \in P/P^{(2)}$, which is generated as $S/P$-module by the images of the $p_i$, so that $\mu (\zeta)\geq{\rm min}(\mu (p_i))=\mu(p_1)$ and finally $\mu(\zeta)=\mu(p_1)$.
We now proceed by induction and assume that the $\mu$-initial forms of $p_1,\ldots ,p_j$ are the first $j$ generators of ${\rm gr}_\mu S$ whose $\mu_1$-value is $>0$. Since $S$ is complete, we may assume that ${\rm in}_{\mu}p_{j+1}$ is not the initial form of an element of the ideal of $S$ generated by $p_1,\ldots ,p_j$; if it were, we might replace $p_{j+1}$ by $p_{j+1}-\sum_{k=1}^ja_qp_q$ without affecting its role as generator of $P$ while increasing the $\mu$-value by an appropriate choice of the $a_q$, and continue until we either reach a contradiction with the minimality of the system of generators or obtain a generator with the desired property. This is essentially the same as the argument in the proof of lemma 4.6 of \cite{Te3}. Therefore the $\mu$-initial form of $p_{j+1}$ must involve a generator of the graded algebra of degree greater than $\mu (p_j)$, say $\overline\zeta_\ell$, so that ${\rm in}_\mu p_{j+1}=\lambda\overline\xi^\alpha\overline\zeta^e_\ell,\ \lambda\in k^*$, with $e\in \N$ and $\overline\xi^\alpha$ is a monomial in generators of ${\rm gr}_\mu S$ different from $\zeta_\ell$. Unless we have ${\rm in}_\mu p_{j+1}=\lambda\overline\zeta_\ell$, this implies that for a representative $\zeta_\ell\in P$ of $\overline\zeta_\ell$, we have $\mu(\zeta_\ell)<\mu (p_{j+1})$. On the other hand we can write $\zeta_\ell=\sum_{k=1}^ja_qp_q+\sum_{k=j+1}^sa_qp_q$ and the inequality we have just seen implies that the $\mu$-initial form of $\zeta_\ell$ must come from the first sum. But then the $\mu$-initial form of $p_{j+1}$ is the initial form of an element of the ideal of $S$ generated by $p_1,\ldots ,p_j$, and we have a contradiction. This proves that the $\mu$-initial form of $p_{j+1}$ is part of a minimal system of generators of ${\rm gr}_\mu S$. On the other hand, the smallest generator of the semigroup of $\mu$ which is not in the semigroup $\langle\Gamma,{\rm in}_\mu p_1,\ldots ,{\rm in}_\mu p_j\rangle$ must, by the same argument as above, have a value $\geq\mu(p_{j+1})$, so that it is $\zeta_\ell$. Assertion 2) now follows by induction.\par\noindent To prove 3) we notice that among the elements $p_i$, those which minimally generate $PS_P$ cannot be in $P^{(2)}$, and that is the case for $p_1,\ldots p_{n-d}$ as we defined them because of the inequalities on valuations.
   \end{proof}
    \begin{lemma}\label{upgen} Given a finite subset $(\xi_i)_{i\in J}$ of a set $(\xi_i)_{i\in I}$ of representatives of the generators $(\overline\xi_i)_{i\in I}$ of the graded algebra ${\rm gr}_\nu R$, chosen in such a way that the valuative Cohen theorem is valid for them, and where $J$ is such that the $(\xi_i)_{i\in J}$ generate the maximal ideal $m_R$ of $R$, let us denote by $(\eta_i)_{i\in J}$ representatives in $S$ of the $(\xi_i)_{i\in J}$. The $(\eta_i)_{i\in J}$ generate the maximal ideal $m_S$ of $S$.
 \end{lemma}
 \begin{proof} Let $M$ denote the ideal of $S$ generated by the $(\eta_i)_{i\in J}$, $m_R$ and $m_S$ the maximal ideals of $R$ and $S$ respectively. By construction me have $M+P=m_S$ since $S/M+P=R/m_R=k$. But since the dimension of the regular ring $S$ is the embedding dimension of $R$, we must have the inclusion $P\subset m_S^2$. By Nakayama's lemma this implies $M=m_S$.
 \end{proof}
Given the rational valuation $\mu$ of rational rank $r+1< n$ on the power series ring $S=k[[x_1,\ldots, x_n]]$ which we have just built, the generators $\gamma_i$ of the semigroup $\Gamma$ are part of the minimal system of generators of the semigroup $\Delta$ of $\mu$. The other elements of the minimal system of generators of $\Delta$ are the $\mu(p_1),\ldots ,\mu(p_s)\in \Z\oplus\Phi$ by proposition \ref{compose}, and the other members of this minimal system of generators of $\Delta$ which are $>\mu(p_s)$  and which we denote by $\delta_a$.\par\noindent
We therefore have a surjection
$$k[(U_i)_{i\in I}, V_1,\ldots ,V_s, (W_a)_{a\in A}]\rightarrow {\rm gr}_\mu S, \ U_i\mapsto {\rm in}_\mu\eta_i, V_j\mapsto {\rm in}_\mu p_j, W_a\mapsto {\rm in}_\mu\zeta_a.$$
 By the valuative Cohen theorem we can lift the $\xi_i\in R$ to elements $\eta_i\in S$ and choose representatives $\zeta_a$ in $S$ for the elements of ${\rm gr}_\mu S$ of degree $\delta_a$ in such a way that the theorem applies, the $p_j$ being finite in number can be kept, and so we obtain a continuous surjective map of topological $k$-algebra $$\Pi\colon \widehat{k[(u_i)_{i\in I},v_1,\ldots ,v_s,(w_a)_{a\in A}]}\rightarrow S,\ u_i\mapsto \eta_i,\ v_j\mapsto p_j,\ w_a\mapsto \zeta_a,$$
to which we can apply Theorem \ref{regular} and obtain the following key
\begin{lemma}\label{goodinit} Let $\theta\in S$ be an element whose initial form in ${\rm gr}_\mu S$ is part \break of a minimal system of generators of this $k$-algebra and let \break$z\in ((u_i)_{i\in I},(v_j)_{1\leq j\leq s},(w_a)_{a\in A})$ be the corresponding variable. A topological generator $F_z$ of the kernel of $\Pi$ which contains $-z$ as a term according to Theorem \ref{regular} can be chosen so that the following holds: If $\theta$ is one of the $\eta_i$ or one of the $p_j$, the initial binomial of $F_z$ involves only the variables $(u_i)_{i\in I}$. If $\theta$ is one of the $\zeta_a$, each term of the initial binomial of $F_z$ involves some of the $v_j$.
\end{lemma} 
\begin{proof} If $\theta$ is one of the $\eta_i$ the result is clear, since the variables in the initial binomial must be of value less than that of $\eta_i$. If $\theta$ is one of the $p_j$, say $p_\ell$, there cannot even exist a $F_z$ whose initial binomial is not in $k[(u_i)_{i\in I}]$ because if that was the case, both terms of the initial binomial would contain a monomial in the $p_j$ and then $p_\ell$ would be in the ideal generated by the $p_j$ of smaller value, which contradicts the minimality of the system of generators. Now if $w_a$ corresponds to $\zeta_a\in P$ which is not one of the $p_j$, there exist $a_i\in \widehat{k[(u_i)_{i\in I},v_1,\ldots ,v_s,(w_a)_{a\in A}]}$ such that $w_a-\sum_{j=1}^sa_jv_j$ is in ${\rm ker}\Pi$ and is not zero. The process of expressing this difference as a combination 
$$w_a-\sum_{j=1}^sa_jv_j=\sum A_\ell F_\ell$$ of the topological generators of ${\rm ker}\Pi$ begins with expressing the initial form for the weight of that difference as a combination of binomials which are part of a generating system for ${\rm ker}{\rm gr}\Pi$. In a minimal such expression there can be no term which is a monomial in the $U_i$ because there is no such term in the initial form of the difference. Since the values can only increase from there and in view of the form of the valuation $\mu$, the sum $\sum A_\ell F_\ell$ can contain no term using only the $u_i$. On the other hand, since $w_a$ is a term in the difference it must appear as a term in that sum, and we know that one of the $F_\ell$ must contain $w_a$ (or $-w_a$) as a term. It follow that one of the $F_\ell$ appearing in the sum, which by construction has an initial binomial which uses some of the $v_j, w_b$, contains $w_a$ (or $-w_a$)  as a term. 
\end{proof}
Keeping the notations of this section, let us now denote by $\Pp$ the closure of the ideal of $\widehat{k[(u_i)_{i\in I},v_1,\ldots ,v_s,(w_a)_{a\in A}]}$ generated by the $(v_j)_{j=1,\ldots ,s}$ and the $(w_a)_{a\in A}$. We have a commutative diagram:
 \[\xymatrix { 
       \widehat{k[(u_i)_{i\in I},v_1,\ldots ,v_s,(w_a)_{a\in A}]}\ar[d]  \ar[r]^{\ \ \ \ \ \ \ \ \ \ \ \ \ \ \ \ \  \Pi}  &S \ar[d] \\
               \widehat{k[(u_i)_{i\in I}]} \ar[r]^{\pi}&  R }\]
               where the left vertical arrow is the quotient by the ideal $\Pp$. It follows from Lemma \ref{goodinit} that we can now prove:
\begin{theorem}{\rm (Structure of rational valuations)}\label{structure} Let $R$ be a complete equicharacteristic noetherian local domain and let $k\subset R$ be a field of representatives of the residue field of $R$. Let $\nu$ be a rational valuation of $R$ and let $(\xi_i)_{i\in I}$ be representatives in $R$ of a minimal set of generators $(\overline \xi_i)_{i\in I}$ of the $k$-algebra ${\rm gr}_\nu R$ for which the valuative Cohen theorem holds. Let $J\subset I$ be a minimal set such that the $(\xi_i)_{i\in J}$ generate the maximal ideal of $R$ and let $(U^{m_\ell}-\lambda_\ell U^{n_\ell})_{\ell\in L}$ be a minimal set of generators of the kernel of the surjective map of $k$-algebras $k[(U_i)_{i\in I}]\to {\rm gr}_\nu R,\ \ U_i\mapsto\overline\xi_i$.\par
Then a set of topological generators for the kernel of the continuous surjective map $\pi\colon\widehat{k[(u_i)_{i\in I}]}\to R$ given by the valuative Cohen theorem is as follows: 
\begin{itemize}
\item For each $i\in I\setminus J$, a generator of the form $$F_i=u^{m^i}-\lambda_i u^{n^i}-g_i((u_q)_{q<i})-u_i,$$ where the initial binomial is one of the  $(U^{m_\ell}-\lambda_\ell U^{n_\ell})_{\ell\in L}$, the series $g_i((u_q)_{q<i})\in \widehat{k[(u_i)_{i\in I}]}$ depends only on variables of index $<i$ and the weight of each of its terms is $>w(u^{m^i})$ and $<w(u_i)$.
\item A finite subset of the complement in $L$ of the set of binomials used above is in bijection with generators of the form $$F_q=u^{m^q}-\lambda_q u^{n^q}-g_q(u),$$ with $g_q(u) \in \widehat{k[(u_i)_{i\in I}]}$ and the weight of each of its terms is $>w(u^{m^q})$.
\end{itemize}
Moreover, each of these generators uses only finitely many of the variables $u_i$.
\end{theorem}
\begin{proof}  From the diagram above we see that we have the equality $\Pi^{-1}(P)={\rm ker}\Pi+\Pp$ and that ${\rm ker}\pi$ is the image of this ideal in $\widehat{k[(u_i)_{i\in I}]}$. So we obtain topological generators of ${\rm ker}\pi$ by reducing the topological generators of ${\rm ker}\Pi$ modulo $\Pp$. By Lemma \ref{goodinit}, those which involve a $-w_a$ as a term vanish entirely,  those which involve a $-v_q$ as a last term do not vanish since their initial binomial involves only the $u_i$ and become the $F_q$ of the theorem. Those which involve a $-u_i$ as a last term pass to the quotient without modification and become the $F_i\in {\rm ker}\pi$.    
\end{proof}
\begin{remark}1) One cannot expect that in general the initial binomials of the series $F_i, F_q$ are all the $(U^{m^\ell}-\lambda_\ell U^{n^\ell})_{\ell\in L}$. As in the regular case, there is no reason why these topological generators should be a (generalized) standard basis.\par\noindent
2) We know of course that we can write each $p_j$ as a series in the $\tilde\eta$. \textit{The point is to write it using only elimination through equations whose initial forms are binomials belonging to our set of generators of the kernel of ${\rm gr}_w\Pi$. This is what will allow us to obtain a quotient of $R$ as an overweight deformation of a finite number of these binomials.}\par\noindent
3) The theorem shows the persistence of the abyssal phenomenon of Remark \ref{choice} even when $R$ is not regular.\par\noindent
4) An interpretation of the theorem is that it presents the valuation $\nu$ on $R$ as the image of the \textit{monomial} valuation on $\widehat{k[(u_i)_{i\in I}]}$ given by the weights, in the sense that the value of an element of $R$ is the maximum weight of its representatives in $\widehat{k[(u_i)_{i\in I}]}$ (see \cite[Proposition 3.3]{Te3}). \par\noindent
5) The semigroup $\Delta$ of values of the valuation $\mu$ on $S$ requires more generators than the $(0,\gamma_i)$, the $\mu(p_j)$ and the $(f,\gamma_i),\ f\in\N, f\geq 1$, and their description would be interesting. We do not need it here thanks to Lemma \ref{goodinit}.
\end{remark}
\end{subsection}
\begin{subsection}{Equations for the toric degeneration}
We can now write the equations for the toric degeneration of $R$ to ${\rm gr}_\nu R$ in the natural parameters, the $v^\phi, \phi>0$.\par
Recall from \cite[Section 2.3]{Te1} that the algebra encoding the toric degeneration of the ring $R$ to its associated graded ring ${\rm gr}_\nu R$ is the valuation algebra $$\Aa_\nu(R)=\bigoplus_{\phi\in\Phi} \Pp_\phi(R)v^{-\phi}\subset R[v^\Phi].$$
Having fixed a field of representatives $k\subset R$ of the residue field, the composed injection $k[v^{\Phi_{\geq 0}}]\subset R[v^{\Phi_{\geq 0}}]\subset \Aa_\nu(R)$ stemming from the fact that $\Pp_\phi(R)=R$ for $\phi\leq 0$ gives $\Aa_\nu(R)$ the structure of a faithfully flat $k[v^{\Phi_{\geq 0}}]$-algebra and according to \cite[Proposition 2.2]{Te1}, the map $${\rm Spec}\Aa_\nu (R)\rightarrow {\rm Spec}k[v^{\Phi_{\geq 0}}]$$ has special fiber ${\rm Spec}{\rm gr}_\nu R$ and its general fiber is isomorphic to $R$ in the sense that, denoting by $(v^{\Phi_+})$ the multiplicative subset of $R[v^{\Phi_{\geq 0}}]$ consisting of elements $v^\phi,\ \phi>0$, the natural inclusion $R[v^{\Phi_{\geq 0}}]\subset \Aa_\nu (R)$ extends to an isomorphism after localization:
$$(v^{\Phi_+})^{-1}R[v^{\Phi_{\geq 0}}]\simeq (v^{\Phi_+})^{-1}\Aa_\nu(R).$$
Let us consider the $k[v^{\Phi_{\geq 0}}]$-algebra $k[v^{\Phi_{\geq 0}}]\widehat{[(\tilde u_i)_{i\in I}]}=k[v^{\Phi_{\geq 0}}]\otimes_k\widehat{k[(\tilde u_i)_{i\in I}]}$, which is the same construction as $\widehat{k[(u_i)_{i\in I}]}$ in section \ref{reminder} but where the series have coefficients in $k[v^{\Phi_{\geq 0}}]$. The variable $\tilde u_i$ still has weight $\gamma_i$. Remember from \cite[Section 2.4]{Te1} that we can define on the algebra $\Aa_\nu(R)$ a valuation $\nu_\Aa$ by $$\nu_\Aa(\sum x_\phi v^{-\phi})={\rm min}( \nu(x_\phi)).$$
\begin{lemma}\label{dyngen}
With the notations of Theorem \ref{cohen}, the $(\xi_iv^{-\gamma_i})_{i\in I}$ constitute a minimal system of homogeneous topological generators of the $k[v^{\Phi_{\geq 0}}]$-algebra $\Aa_\nu (R)$ in the sense that the map of $k[v^{\Phi_{\geq 0}}]$-algebras
$$\tilde \pi\colon k[v^{\Phi_{\geq 0}}]\widehat{[(\tilde u_i)_{i\in I}]}\longrightarrow \Aa_\nu(R),\ \ \tilde u_i\mapsto \xi_iv^{-\gamma_i}$$
is surjective and continuous for the topologies defined by the weight and the valuation $\nu_\Aa$ respectively.
\end{lemma}
\begin{proof} By the valuative Cohen theorem (Theorem \ref{cohen}) if $x\in\Pp_\phi(R)$ we can write $x=\sum_e d_e\xi^e$ where the $d_e\in k^*$ and the $\xi^e$ are monomials in the $\xi_i$ of increasing valuations, with the minimum value being $\geq \phi$. Let us write $-\nu(x)=-\nu(\xi^e)+\eta_e$ with $\eta_e\geq 0$. Note that the $\eta_e$ corresponding to the smallest index in the sum is zero if $\nu(x)=\phi$ and all the others are $>0$. Then we can write $xv^{-\phi}=\sum_ed_ev^{\eta_e}\xi^ev^{-\nu(\xi^e)}$ and since the $\xi^ev^{-\nu(\xi^e)}$ are monomials in the $\xi_iv^{-\gamma_i}$ and the minimality is clear as well as the continuity, this proves the result. \end{proof}
We now think of the $\tilde u_i=\xi_iv^{-\gamma_i}$ as defining an embedding of the total space of the toric degeneration into the space corresponding to the algebra $k[v^{\Phi_{\geq 0}}]\widehat{[(\tilde u_i)_{i\in I}]}$. We are seeking equations for the image of this embedding, that is, generators for the kernel of the map $\tilde \pi$ of Lemma \ref{dyngen} above.\par\noindent
If a series $H((\tilde u_i))_{i\in I}$ is in the kernel, we see that by definition of the map, we have $H((v^{\gamma_i}\tilde u_i)_{i\in I})=0$ in $R$, and conversely. so that the change of variables $u_i\mapsto v^{\gamma_i}\tilde u_i$, which is well defined over $(v^{\Phi_+})^{-1}k[v^{\Phi_{\geq 0}}]$, carries the kernel of the map $\pi$ of Theorem \ref{cohen} to the kernel of $\tilde\pi$. It realizes the isomorphism, valid over $(v^{\Phi_+})^{-1}k[v^{\Phi_{\geq 0}}]$  of the fiber of ${\rm Spec}\Aa_\nu (R)\to {\rm Spec}k[v^{\Phi_{\geq 0}}]$ over a point where "all the $v^\phi,\ \phi>0$ are $\neq 0$" with the space corresponding to $R$.\par\noindent  Now for each topological generator $F_\ell=u^{m^\ell}-\lambda_\ell u^{n^\ell}+\cdots$ of the kernel of $\pi$ we see that $F_\ell (v^\gamma \tilde u)$, with the obvious contracted notation, is divisible, as an element of $k[v^{\Phi_{\geq 0}}]\widehat{[(\tilde u_i)_{i\in I}]}$, by $v^{w(m^\ell)}=v^{w(n^\ell)}$. The ideal of $k[v^{\Phi_{\geq 0}}]\widehat{[(\tilde u_i)_{i\in I}]}$ topologically generated by the $$\tilde F_\ell=v^{-w(m^\ell)}F_\ell(v^\gamma\tilde u)=\tilde u^{m^\ell}-\lambda_\ell \tilde u^{n^\ell}+\cdots$$ is contained in the kernel of $\tilde\pi$, is equal to it and also isomorphic to the kernel of $\pi$ over $(v^{\Phi_+})^{-1}k[v^{\Phi_{\geq 0}}]$, while the space it defines modulo the ideal $(v^{\Phi_{>0}})$ is the same as the space defined by $\Aa_\nu(R)$ modulo that ideal, namely ${\rm Spec}{\rm gr}_\nu R$. By flatness of $\Aa_\nu(R)$ over $k[v^{\Phi_{\geq 0}}]$, they have to coincide.\par\medskip
Combining this with Theorem \ref{structure} we have obtained the following:
\begin{theorem}
With the notations of Theorem \ref{structure}, a set of topological generators in $k[v^{\Phi_{\geq 0}}]\widehat{[(\tilde u_i)_{i\in I}]}$ for the ideal defining the total space of the degeneration of $R$ to ${\rm gr}_\nu R$ is as follows:
\begin{itemize}
\item For each $i\in I\setminus J$, a generator of the form $$\tilde F_i=\tilde u^{m^i}-\lambda_i \tilde u^{n^i}-\tilde g_i(v,(\tilde u_q)_{q<i})-v^{\gamma_i-w(u^{m^i)}}\tilde u_i,$$ where the initial binomial is one of the  $(U^{m_\ell}-\lambda_\ell U^{n_\ell})_{\ell\in L}$, the series $\tilde g_i(v,(\tilde u_q)_{q<i})\in k[v^{\Phi_{\geq 0}}]\widehat{[(\tilde u_i)_{i\in I}]}$ depends only on variables of index $<i$ and the weight of each of its terms is $>w(u^{m^i})$ and $<w(u_i)$.
\item A finite subset of the complement in $L$ of the set of binomials \break $(\tilde u^{m_\ell}-\lambda_\ell \tilde u^{n_\ell})_{\ell\in L}$ used by the $\tilde F_i$ is in bijection with generators of the form $$\tilde F_q=\tilde u^{m^q}-\lambda_q \tilde u^{n^q}-\tilde g_q(v,\tilde u),$$ with $\tilde g_q(v,\tilde u) \in k[v^{\Phi_{\geq 0}}]\widehat{[(\tilde u_i)_{i\in I}]}$ and the weight of each of its terms is $>w(\tilde u^{m^q})$.
\end{itemize}
Each series uses finitely many of the variables $\tilde u_i$ and each term appearing in the series $\tilde F_i$ or $\tilde F_q$ is of the form $c_ev^{w(\tilde u^e)-w(u^{m^\ell})}\tilde u^e$ or $c_ev^{w(\tilde u^e)-w(u^{m^q})}\tilde u^e$ with $c_e\in k^*$.
\end{theorem}
\begin{example}
\end{example} Consider the overweight deformation, as in section \ref{ow}: $$y^2-x^3-u, \ u^2-x^{s-1}y$$ of the binomial ideal $y^2-x^3, u^2-x^{s-1}y$, where $x,y,u$ have weights $4,6,2s+1$ respectively, with $s\geq 6$. Set $\tilde x=v^{-4}x,\tilde y=v^{-6}y, \tilde u=v^{-(2s+1)}u$. \par\noindent Substituting $x=v^4\tilde x,y=v^6\tilde y,u=v^{2s+1}\tilde u$ in the equations, we obtain $v^{12}\tilde y^2-v^{12}\tilde x^3-v^{2s+1}\tilde u=0, v^{4s+2}\tilde u^2-v^{4s+2}\tilde x^{s-1}\tilde y=0$. \par After dividing each equation by the highest possible power of $v$, we obtain the equations for the total space of the toric degeneration over ${\rm Spec}k[v]$ of the formal curve in $\A^3(k)$ with ring $R=k[[x,y,u]]/(y^2-x^3-u,u^2-x^{s-1}y)$ to the monomial (toric) curve $\tilde x=t^4, \tilde y=t^6, \tilde u=t^{2s+1}$: $$\tilde y^2-\tilde x^3-v^{2(s-6)+1}\tilde u=0, \ \tilde u^2-\tilde x^{s-1}\tilde y=0.$$
\end{subsection}

\end{section}
\begin{section}{\small{The basic example of approximation of a valuation by Abhyankar semivaluations}}\label{ex}

Let $k$ be an algebraically closed field and let $\nu$ be a rational valuation of rational rank one on the ring $R=k[[x,y]]$. It cannot be Abhyankar, and therefore the semigroup $\Gamma=\nu (R\setminus\{0\})$ is not finitely generated (see \cite[beginning of \S 7]{Te3}). Up to a change of variables we may assume that $\nu(x)<\nu(y)$ and that $\nu(y)$
is the smallest non zero element of $\Gamma$ which is not in $\Gamma_0=\N\nu(x)$. By \cite[Theorem 1.1 and Lemma 2.1]{C-V}, the semigroup $\Gamma$ is generated by positive rational numbers $\nu(x),\nu(y),(\gamma_i)_{i\geq 2}$ and we can normalize the valuation by setting $\gamma_0=\nu(x)=1$. Here, as in Example \ref{ex1}, we denote the smallest element of $\Gamma$ by $\gamma_0$ instead of $\gamma_1$ to underscore the special role played historically by the coordinate $x$.\par\noindent Moreover, denoting by $\Phi_i$ (resp. $\Gamma_i$) the subgroup (resp. subsemigroup)  of $\Q$ generated by $1,\nu(y),\nu(u_2),\ldots ,\nu(u_i)$, with $\Phi_0=\Z$, $\Phi_1$ (resp. $\Gamma_1$) generated by $\nu(x),\nu(y)$, and denoting by $n_i$ the index $[\Phi_i:\Phi_{i-1}]$, we have that $\gamma_{i+1}>n_i\gamma_i$ and as a consequence $n_i\gamma_i\in \Gamma_{i-1}$ and the relations between the $\gamma_i$ are generated by the expressions of this last result as $$n_i\gamma_i= t^{(i)}_0\gamma_0+t^{(i)}_1\gamma_1+\sum_{q=2}^{i-1}t^{(i)}_q\gamma_q\ \ {\rm with}\ t^{(i)}_q\in\N,\ \ 0\leq q\leq i-1,$$
where we may assume that $t^{(i)}_q<n_q$ for $q\geq 1$.\par\noindent
According to Theorem \ref{regular}, by a good choice of the representatives $\xi_i\in k[[x,y]]$ of the generators $\gamma_i$ the map 
\[\widehat{k[x,y,(u_i)_{i\geq 2}]}\to k[[x,y]],\ x\mapsto x, y\mapsto y, u_i\mapsto \xi_i,\] is well defined and surjective, and
has a kernel generated up to closure by $$\begin{array}{lr}y^{n_1}-x^{t^{(1)}_0}-g_1(x,y)-u_3,\\ \vdots \\ u_i^{n_i}-x^{t^{(i)}_0}y^{t^{(i)}_1}\prod_{2\leq q\leq i-1} u_q^{t^{(i)}_q}-g_i(x,y,(u_q)_{2\leq q\leq i})-u_{i+1}\\ \vdots \end{array}$$ with $g_i(x,y,(u_q)_{2\leq q\leq i})\in k[[(x,y,(u_q)_{2\leq q\leq i}]]$ satisfying the overweight condition and having each term of weight $<\gamma_{i+1}$.\par
Now if we keep only the first $i$ equations, and set $u_j=0$ for $j\geq i+1$, we obtain by elimination of the variables $u_q$ , $2\leq q\leq i$ the equation $p_i(x,y)=0$ of a plane branch $C_i$ in the formal affine plane. It has a unique valuation, which is Abhyankar and has a semigroup $\Gamma_i$ generated by $1,\nu(y),\nu(u_q)_{3\leq q\leq i}$. The unique valuation on the branch $C_i$ is an Abhyankar semivaluation on $k[[x,y]]$ and the corresponding semigroups $\Gamma_i=<1,\gamma_1,\gamma_2,\ldots ,\gamma_i>$ fill the semigroup $\Gamma$ as $i$ grows. The intersection numbers with $C_i$ at the origin of the elements $\xi_q\in k[[x,y]]$ are the $\gamma_q$, for $2\leq q\leq i$. Finally, by Chevalley's theorem, there exists a function $\beta\colon\N\to \N$ with $\beta(i)$ tending to infinity with $i$, such that $p_i(x,y)\in (x,y)^{\beta(i)}$.\par
To the branch $C_i$ is associated a rank two valuation $\nu_i$ on $k[[x,y]]$ defined as follows: each element $h\in k[[x,y]]$ can be written uniquely as $h=p_i^n\tilde h$ with $\tilde h$ not divisible by $p_i$ and then the restriction of $\tilde h$ to $C_i$ has a value $\gamma\in \Gamma_i$.\par\noindent The map $h\mapsto (n,\gamma)\in \Z\times \Phi_i$ 
is the rank two valuation $\nu_i$, which is a rational Abhyankar valuation dominating $k[[x,y]]$. The valuations $\nu_i$ approximate $\nu$ in the sense that given any $h\in k[[x,y]]$, for sufficiently large $i$ the element  $p_i(x,y)\in (x,y)^{\beta(i)}$ does not divide $h$ so that $\nu_i(h)\in \{0\}\times \Gamma_i$ is also $\nu(h)$ by what we have seen above. So we can indeed approximate the non-Abhyankar valuation $\nu$ of rank one and rational rank one by Abhyankar valuations of $k[[x,y]]$ of rank two and rational rank two, but the fundamental fact is the approximation of $\nu$ by Abhyankar semivaluations.
\begin{remark} If the algebraically closed field $k$ is of characteristic zero, each equation $p_i(x,y)=0$ has for roots Puiseux expansions with $i+1$ characteristic exponents $\beta_0,\ldots,\beta_i$ which coincide, up to renormalization, with Puiseux expansions up to and excluding the $(i+1)$-th Puiseux exponents of the roots of the equations $p_j=0$ for $j>i$ (see \cite{PP}). As $i$ tends to $\infty$ these Puiseux series converge to a series $y_\infty(x)\in k[[x^{\Q_{\geq 0}}]]$ which is not a root of any polynomial or power series in $x,y$ because the denominators of the exponents are not bounded. Given $h(x,y)\in k[[x,y]]$, associating to it the order in $x$ of $h(x,y_\infty(x))$ defines a valuation, which is the valuation $\nu$. For more details in a more general situation, see \cite{M}.\par\noindent The relationship between these Puiseux exponents $\beta_q$ and the generators of the semigroup is quite simple (see \cite[Chap.1]{Z2}): it is $\gamma_{q+1}-n_q\gamma_q=\beta_{q+1}-\beta_q$. These beautiful facts are unfortunately missing when one tries to understand valuations over fields of positive characteristic.  Even for branches the behaviour of solutions in generalized power series $y(x)\in k[[x^{\Q_{\geq 0}}]]$ of algebraic equations is much more complicated when $k$ is of positive characteristic (see \cite{K}) and as far as we know the relationship of these solutions with the semigroup of values has not been clarified (see \cite[Section 1]{Te4}).\end{remark}

\end{section}
\begin{section}{Approximating a valuation by Abhyankar semivaluations}
The purpose of this section is to show that the situation described in Example \ref{ex} is in fact quite general. We start from Theorem \ref{structure}. Let us choose a finite initial subset $B_0$ of $I$ which contains:\begin{itemize}
\item The set $J$ indices of the elements $(\xi_i)_{i\in J}$ minimally generating the maximal ideal of $R$;
\item A set of $r={\rm rat.rk.}\Phi$ indices $i_1,\ldots ,i_r$ such that the $(\gamma_{i_t})_{t=1,\ldots ,r}$ rationally generate the group $\Phi$;
\item The indices of the finite set of variables $u_i$ appearing in the equations $(F_q)_{q=1,\ldots ,s}$ of Theorem \ref{structure}.
\end{itemize}
Let $B$ be a finite initial subset of $I$ containing $B_0$ and consider in $R$ the ideal $K_B$ generated by the $(\xi_i)_{i\notin B}$ and the images by the map $\pi$ of the $F_\ell$ whose initial binomial involves only variables whose index is in $B$ and in which all the $u_i$ with $i\notin B$ have been set equal to $0$. Note that only the $F_\ell$ which use only variables in $B$ are mapped to zero in this operation. 
\begin{proposition}\label{quot}
\begin{enumerate}
\item The quotient $R/K_B$ is an overweight deformation of the prime binomial ideal generated by the binomials $u^{m^\ell}-\lambda_\ell u^{n^\ell}$ which are contained in $k[(u_i)_{i\in B}]$. In particular, $K_B$ is a prime ideal.
\item Given an integer $D$ there exists a finite initial set $B(D)$ such that $K_{B(D)}\subset m^{D+1}$.
\end{enumerate}
\end{proposition}
\begin{proof}
The following result is classical, see \cite[Lemma 1-90]{G-L-S} for the complex analytic avatar:

 \begin{lemma}\label{dim} Let $R_0\subset R$ be an injection of complete local noetherian $k$-algebras with residue field $k$. Then we have ${\rm dim}R\geq {\rm dim}R_0$.
 \end{lemma}
 By the formal normalization lemma, there exist elements $x_1,\ldots ,x_e$ in $R\setminus m_0R$ such that the induced map of $k$-algebras $k[[X_1,\ldots ,X_e]] \to R/m_0R$ sending $X_i$ to the image of $x_i$ is injective and finite.\par\noindent By the Weierstrass preparation theorem in the form "quasi-finite implies finite", the map $R_0[[X_1,\ldots ,X_e]]\to R$ sending $X_i$ to $x_i$ is also injective and finite, so that ${\rm dim} R={\rm dim} R_0+e\geq {\rm dim} R_0$. \par\medskip\noindent

 \textit{In the following lemma, an} overweight unfolding \textit{is a deformations of equations which adds only terms of higher weight, without any condition on the initial forms of the elements of the ideal generated by the deformed equations.}
 \begin{lemma}\label{flat} Let $F_0=(u^{m^\ell}-\lambda_\ell u^{n^\ell})_{\ell \in L}$ be a prime binomial ideal in $k[[u_1,\ldots ,u_N]]$ corresponding to a system of generators of the relations between the generators of a finitely generated semigroup $\Gamma=\langle \gamma_1,\ldots ,\gamma_N\rangle$ generating a group of rational rank $r$. Let $w$ be the weight  on $k[[u_1,\ldots ,u_N]]$ defined by giving $u_i$ the weight $\gamma_i$.\par Let $(F_\ell=u^{m^\ell}-\lambda_\ell u^{n^\ell} +\sum_{w(u^p)>w(u^{m^\ell})}c^{(\ell)}_pu^p)_{\ell\in L}$ be an overweight unfolding of the binomials. Let $F$ be the ideal of $k[[u_1,\ldots ,u_N]]$ generated by the $(F_\ell)_{\ell\in L}$ and $R=k[[u_1,\ldots ,u_N]]/F$.\par\noindent If ${\rm dim}R\geq r$, then ${\rm dim}R= r$ and the unfolding is an overweight deformation of the binomial ideal.
 \end{lemma}
 Exactly as in the proof of Theorem 3.3 of \cite{Te3}, we can define an order function with values in $\Gamma$ on $R$ by associating to a non zero element of $R$ the highest weight of its preimages in $k[[u_1,\ldots ,u_N]]$. Then we have, as in \cite[subsection 2.3]{Te1},
 a faithfully flat specialization of $R$ to $k[[u_1,\ldots ,u_N]]/{\rm in}_wF$, where ${\rm in}_wF$ is the $w$-initial ideal of the ideal $F$, and ${\rm dim}(k[[u_1,\ldots ,u_N]]/{\rm in}_wF)={\rm dim}R$. On the other hand we have the exact sequence
 \[(0)\rightarrow {\rm in}_w F/F_0\rightarrow k[[u_1,\ldots ,u_N]]/B\rightarrow k[[u_1,\ldots ,u_N]]/{\rm in}_wF\rightarrow(0),\]
 and since the domain $ k[[u_1,\ldots ,u_N]]/F_0$ is of dimension $r$ (see \cite[Proposition 3.7]{Te1}), the inequality ${\rm dim}R\geq r$ is equivalent to the inequality \[{\rm dim}k[[u_1,\ldots ,u_N]]/{\rm in}_wF\geq {\rm dim}k[u_1,\ldots ,u_N]/F_0\] which implies the equality of dimensions and, since $F_0$ is prime, the equality ${\rm in}_wF=F_0$ which means that we have an overweight deformation.

  By construction, the ring $R/K_B$ is a quotient of $k[[(u_i)_{i\in B}]]$ by the ideal generated by the $ F_\ell\vert B$ which are the $F_\ell$ whose initial forms are in $k[(u_i)_{i\in B}]$ and where we have set $u_i=0$ if $i\notin B$. So it is an overweight unfolding of the ideal generated by its initial binomials. The quotient of $k[(u_i)_{i\in B}]$ by this binomial ideal is isomorphic by Proposition \ref{smgrp} to the semigroup algebra $k[t^{\Gamma_i}]$ where $\Gamma_i$ is the semigroup generated by the $(\gamma_i)_{i\in B}$ and so has dimension $r$.\par  By Lemmas \ref{dim} and \ref{flat}, to prove the first assertion of the proposition it suffices to prove that $R/K_B$ contains the power series ring $k[[u_{i_1},\ldots ,u_{i_r}]]$. But $(i_1,\ldots ,i_r)\subset B$ by our choice of $B_0$ and a series $h(u_{i_1},\ldots ,u_{i_r}) \in k[(u_i)_{i\in I}]$ cannot belong to the ideal of the $F_\ell$ used to define $K_B$: since the weight of the variables are rationally independant the initial form of $h(u_{i_1},\ldots ,u_{i_r})$ is a monomial and the initial ideal of the $F_\ell$ contains no monomial. \par\noindent Thus, the map $k[[u_{i_1},\ldots ,u_{i_r}]]\to R/K_B$ is injective so that ${\rm dim}R/K_B\geq r$ by Lemma \ref{dim} and we have the result by Lemma \ref{flat},\par\noindent
The second part of the theorem follows from the fact that by Theorem \ref{Chevalley} we can choose $B=B(D)$ so that the $(\xi_i)_{i\notin B}$ are all in $m^{D+1}$ and by construction the image in $R$ of $ F_\ell\vert B\in k[[(u_i)_{i\in B}]]$ differs from zero by a series all of whose terms are in the ideal generated by the $(\xi_i)_{i\notin B}$.
   \end{proof} 
To summarize, the ring  $R/K_B$  is indeed an overweight deformation of the prime binomial generated by the $u^{m^\ell}-\lambda_\ell u^{n^\ell}$ which are in $k[[(u_i)_{i\in B}]]$ in the sense of  subsection \ref{ow} and \cite[Section 3]{Te3}, and the induced valuation has to be Abhyankar, of rational rank $r={\rm rat.rk.}\nu={\rm dim}R/K_B$. The value semigroup of this Abhyankar valuation, which we shall denote by $\nu_B$, is the subsemigroup of $\Gamma$ generated by the $(\gamma_i)_{i\in B}$.\par\noindent
  We now study some properties of the collection of ideals $K_{B}$.\par\medskip
 Notice that given an inclusion $B\subset B'$, the set of generators of the ideal $K_{B'}$ contains less of the variables $u_i$, but in general more of the images of the series $F_\ell$ because the variables with indices in $B'\setminus B$ will in general appear in the initial forms of some new $F_\ell$'s. So we have:
 \begin{remark}\label{order}In general when  $B\subset B'$ there is no inclusion between the ideals $K_B$  and $K_{B'}$. However, it follows from Theorem \ref{Chevalley} that if for some integer $D$ we have $K_B\subset m^{D+1}$ and $B\subset B'$, then $K_{B'}\subset m^{D+1}$. The $m$-adic order of the ideals $K_B$ grows with $B$ and tends to infinity.
 \end{remark}
 \begin{corollary}\label{vanish} We have $\bigcap_{B\supset B_0}K_{B}=(0)$.
 \end{corollary}
 \begin{proof}This follows from the fact that $R$ is noetherian and the second part of the proposition. \end{proof}
 \begin{corollary}\label{jacobian}
 Let $J\subset R$ be the jacobian ideal, defining the singular locus of the formal germ associated to $R$.  There exists a finite initial set $B_J$ containing $B_0$ and such that if an initial set $B$ contains $B_J$ we have that $J\not\subset K_{B}$ so that the localization $R_{K_{B}}$ is a regular ring. In particular, in the sequence of Corollary \ref{approx}, we have that $J\not\subset K_{B_t}$ and $R_{K_{B_t}}$ is a regular ring for large enough $t$.
 \end{corollary}
 \begin{proof} Since $R$ is a domain and a quotient of a power series ring, we know that the ideal $J$ is not zero. Since $R$ is noetherian, we know that $\bigcap_{j=1}^\infty m^j=(0)$. Let $D$ be the largest integer $E$ such that $J\subset m^E$ in $R$. If we take $B_J=B(D)$ as in the proposition, we have that $K_{B}\subset m^{D+1}$. Thus, if $B_J\subset B$ it is impossible for $J$ to be contained in $K_{B}$.\end{proof}
 
\begin{theorem}\label{AbhApprox} The Abhyankar semivaluations $\nu_B$ of $R$ are better and better approximations of the valuation $\nu$ as the finite initial sets $B$ grow in the sense that:
\begin{enumerate} 
\item For any $x\in R\setminus\{0\}$, let $D$ be its $m$-adic order. By Theorem \ref{Chevalley} there are finite initial sets $B$ containing $B_0$ such that $K_B\subset m^{D+1}$. Then for any finite initial set $B'\supset B$ we have $x\notin K_{B'}$ and for any such $B'$ we have the equality $\nu(x)=\nu_{B'}(x)$. In particular, for any $x\in R\setminus\{0\}$ there is an index $t(x)$ in the sequence of Corollary \ref{approx} such that for $t\geq t_0(x)$ we have $x\notin K_{B_t}$ and $\nu(x)=\nu_{B_t}(x)$.
\item The union of the finitely generated semigroups $\Gamma_B=\langle(\gamma_i)_{i\in B}\rangle$ is equal to $\Gamma$, and in particular the nested union of the $\Gamma_{B_t}$ corresponding to the sequence of Corollary \ref{approx} is equal to $\Gamma$.
\end{enumerate}
\end{theorem}
\begin{proof} In view of Corollaries \ref{approx} and \ref{vanish} and of Remark \ref{order} it suffices to show that if $x\in R\setminus K_B$ its $\nu$-value is equal to its $\nu_B$-value. Consider the following diagram, which results from our constructions. 
\[\xymatrix{&k[[(u_i)_{i\in B}]]\ar[drr]_{\pi_B} \ar @{^{(}->}[r]^{\iota_B}&\widehat{k[(u_i)_{i\in I}]} \ar[r]^\pi&  R\ar[d]\\
              & && R/K_B}\]
  The kernel of $\pi_B$ is generated by the $F_\ell\vert B$ as we have seen above. By definition (see \cite[Proposition 3.3]{Te3}) of the Abhyankar valuation on $R/K_B$, the $\nu_B$-value of an element $\overline x\in R/K_B$ is the maximum weight of its representatives in $k[[(u_i)_{i\in B}]]$. Let $h(\overline x)$ be such a representative. If its weight is not the $\nu$-value of its image in $R$ by the map $\pi$, its initial form must be in the binomial ideal generated by the initial forms of the $F_\ell$ (see the sentence after the statement of Theorem \ref{cohen} or \cite[Proposition 3.6]{Te3}). But then it is in the ideal generated by the initial forms of the $F_\ell\vert B$ and we have a contradiction.                  
  \end{proof}
\begin{remark}Just as in Example \ref{ex}, using Corollary \ref{jacobian}, we can, for large enough $B$, compose the valuation $\nu_B$ with the $K_BR_{K_B}$-adic valuation of the regular local ring $R_{K_B}$, to obtain a valuation $\tilde\nu_B$ on $R$ of rational rank $r+1\leq {\rm dim}R$. These valuations are Abhyankar if  $r+1={\rm dim}R$ and approximate $\nu$ in the sense that given $x\in R$, for large enough $B$ we have $\tilde\nu_B (x)=\nu(x)$ since $x\notin K_B$.
\end{remark}
\end{section}

\printindex
\end{document}